\documentclass[preprint,10pt]{elsarticle}

\usepackage{amssymb}
 \usepackage{amsthm}
\usepackage{amsmath}

\newtheorem{thm}{Theorem}[section]
\newtheorem{cor}[thm]{Corollary}

\newtheorem{hypothesis}[thm]{Hypothesis}

\newtheorem{prop}[thm]{Proposition}

\newtheorem{defn}[thm]{Definition}

\newtheorem{rem}[thm]{Remark}

\def\bE{\mathbb{E}}
\def\bU{\mathbb{U}}

\def\cF{\mathcal{F}}

\def\bR{\mathbb{R}}
\def\P{\mathbb{P}}
\def\PP{\mathbb{P}}

\def\dd{{\rm d}}

\def\tn{|\!|\!|\!}
\def\bx{\mathbf{x}}

\def\by{\mathbf{y}}
\def\bJ{\mathbb{J}}
\def\cG{\mathcal{G}}
\def\cL{\mathcal{L}}
\def\cB{\mathcal{B}}
\def\cU{\mathcal{U}}
\def\R{\mathbb{R}}
\def\bN{\mathbb{N}}
\def\cK{\mathcal{K}}

\def\cD{\mathcal{D}}
\def\cQ{\mathcal{Q}}

\def\bE{\mathbb{E}}

\def\cF{\mathcal{F}}

\def\bR{\mathbb{R}}
\def\P{\mathbb{P}}
\def\PP{\mathbb{P}}
\def\bA{\mathbf{A}}

\def\dd{{\rm d}}

\def\tn{|\!|\!|\!}
\def\bx{\mathbf{x}}

\def\by{\mathbf{y}}

\def\bJ{\mathbb{J}}
\def\cG{\mathcal{G}}
\def\cL{\mathcal{L}}
\def\cB{\mathcal{B}}
\def\tF{\tilde{\mathcal{F}}}

\def\tX{\tilde{\mathbf{x}}}

\def\tY{\tilde{Y}}
\def\tW{\tilde{W}}
\def\tZ{\tilde{Z}}
\def\cU{\mathcal{U}}
\def\R{\mathbb{R}}
\def\bN{\mathbb{N}}
\def\bU{\mathbb{U}}
\def\cK{\mathcal{K}}

\def\cF{\mathcal{F}}

\def\cQ{\mathbf{Q}}
\def\bQ{\mathbf{Q}}

\journal{Stochastic Processes and applications}

\begin{document}

\begin{frontmatter}



\title{Infinite horizon Stochastic Optimal Control for Volterra equations with completely monotone kernels}


\author{Elisa Mastrogiacomo}

\address{University of Milano Bicocca, Department of Statistics and Quantitative Methods\\
 Piazza Ateneo Nuovo n 1, 20123 Milano, Italy}

\begin{abstract}
     The aim of the paper is to study an optimal control problem on infinite horizon for an infinite dimensional integro-differential equation with completely monotone kernels, where we assume that the noise enters the system when
we introduce a control. We start by reformulating the state equation into a semilinear evolution
equation which can be treated by semigroup methods. The application to optimal control provide
other interesting result and require a precise description of the properties of the generated semigroup.
The main tools consist in studying the differentiability of the forward-backward system with infinite horizon corresponding with the reformulated problem and the proof of existence and uniqueness of of mild solutions
to the corresponding HJB equation. 
\end{abstract}


\end{frontmatter}
\section{Introduction}\label{sec:intro}
In this paper we study a stochastic optimal control problem with infinite horizon for an infinite dimensional integral equation of Volterra type on a separable Hilbert space.
Our starting point is a controlled stochastic Volterra equation of the form 
\begin{align}\label{eq:Volterra}
    \begin{cases}
       \frac{\dd}{\dd t} \int_{-\infty}^t a(t-s)u(s)\dd s= A u(t)+ f(u(t))\\
         \qquad \qquad \qquad \qquad + g \,[\, r(u(t),\gamma(t))
            + \dot{W}(t)\, ], \qquad t\in [0,T]\\
       u(t)=u_0(t), \qquad t\leq 0.
    \end{cases}
\end{align}
for a process $u$ in a Hilbert space $H$, where $W(t), \ t\geq 0$ is a cylindrical Wiener process defined on a suitable probability space $(\Omega, \cF, (\cF_t)_{t\geq 0},\P) $ with values into a (possibly different) Hilbert space $\Xi$; the kernel $a$ is completely monotonic, locally integrable and
singular at $0$; $A$ is a linear operator which generates an analytical semigroup; $f$ is a Lipschitz continuous function from $H$ into itself; $g$ is 
bounded linear mapping from $H$ into $L_2(\Xi,H)$ (the space of Hilbert-Schmidt operators from $\Xi$ to $H$). The function $r$ is a bounded Borel measurable mapping from $ H \times \cU$ into $\Xi$. Also, the control is modellized by the predictable
process $\gamma$ with values in some specified subset $\cU$ (the set of control actions) of a third Hilbert space $U$. 
We notice that the control enters the system together with the noise. This special structure is imposed by our techniques; however the presence of the operator $r$ allows more generality. 

The optimal control problem that we wish to treat in this paper consists in minimizing the 
functional cost 
\begin{align}\label{eq:cost}
    \bJ(u_0,\gamma)=\bE \int_0^\infty e^{-\lambda t}\ell(u(t),\gamma(t))\dd t,
\end{align}
where $\ell: H \times U \to \R$ is a given real bounded function and  $\lambda$ is any positive number. 

Our first task is to provide a semigroup setting for the problem, by the state space setting first 
introduced by Miller in \cite{miller/1974} and Desch and Miller \cite{desch/miller/1988} and recently revised, for the stochastic case, in Homan \cite{homan}, Bonaccorsi and Desch \cite{bonaccorsi/desch}, Bonaccorsi and Mastrogiacomo \cite{BoMa-JEE}. Within this approach, equation \eqref{eq:Volterra} is reformulated into an abstract evolution equation 
without memory on a different Hilbert space $X$. Namely, we rewrite equation \eqref{eq:Volterra} as
\begin{align}\label{eq:state-eq}
   \begin{cases}
    \dd \bx(t)= B \bx(t)\dd t+ (I-B)P \, f( J \bx(t))\dd t \\
       \qquad \qquad \qquad + (I-B)P g (r(J \bx(t),\gamma(t))\dd t+\dd W(t))\\
     \bx(0)=x,
   \end{cases}
\end{align}
where the initial datum $x$ is obtained from $u_0$ by a suitable transformation. 
 Moreover, $B$ is the infinitesimal generator of an analytic semigroup $e^{tB}$
on $X$. $P: H \to X$ is a linear mapping which acts as a sort of projection into the space 
$X$. $J: \ D(J) \subset X \to H$ is an unbounded linear functional on $X$,
which gives a way going from $\bx$ to the solution to problem \eqref{eq:Volterra}.
In fact, it turns out that $u$ has the representation 
\begin{align*} 
     u(t)=\begin{cases}
           J \bx(t), & t>0,\\
          u_0(t), & t\leq 0.
            \end{cases}
\end{align*}
For more details, we refer to the original papers \cite{bonaccorsi/desch,homan}.
Further, the optimal control problem, reformulated into the state setting $X$, 
consists in minimizing the cost functional
\begin{align*}
     \bJ(x,\gamma)=\bE \int_0^\infty e^{-\lambda t}\ell(J \bx(t),\gamma(t))\dd t 
\end{align*}
(where the initial condition $u_0$ is substituted by $x$ and the process $u$ is substituted by $J\bx$). It follows that $\gamma$ is an optimal control
for the original Volterra equation if and only if it is an optimal control for that state equation
\eqref{eq:state-eq}.  

To our knowledge, this paper is the first attempt to study optimal control problems with infinite horizon for stochastic Volterra equations. In order to handle the control problem we introduce the weak reformulation; in this setting we wish to perform the standard program of synthesis of the optimal control, by proving the so-called fundamental relation for the value function and then solving in the weak sense the closed loop equation. The main step is to study the stationary Hamilton-Jacobi-Bellman equation corresponding with our problem. In our case, this is given by
\begin{align}\label{eq:HJB}\tag{HJB}
         \cL v(x)=\lambda v(x)-\psi(x,\nabla v (x)(I-B) P g(Jx)), \qquad x\in X, 
\end{align} 
where $\cL$ is the infinitesimal generator of the Markov semigroup corresponding to the process
$\bx$: 
\begin{multline*}
     \cL[h](x)=\frac{1}{2}{\rm Tr} [(I-B)Pg\,\nabla^2h(x)g^*\,P^*(I-B)^*]+\\
    + \langle Bx+(I-B)P f(Jx),\nabla h(x)\rangle.
\end{multline*}
and $\psi$ is the Hamiltonian of the problem, defined in terms of $\ell$, i.e.
\begin{align}\label{eq:ham}
      \psi(x,z):= \inf_{\gamma \in \cU}\left\{ \ell(J x,\gamma)+ z\cdot r(x,\gamma)\right\}
\end{align}
Here $\nabla h(x) \in H^\star$ denotes the G\^ateaux derivative at point $x\in H$ and $\nabla^2 h$ is the second G\^ateaux derivative, identified with an element of $L(H)$ (the vector space of linear and continuous functionals on $H$).
Equations of type \ref{eq:HJB} have been studied by probabilistic approach in several 
finite dimensional situations \cite{Buckdan,Darling/Pardoux,Pardoux/Peng/92,Pardoux/Rascanu/99} and extended in infinite dimension by Fuhrman and Tessitore \cite{FuTe/2004}. In particular, an equation like \ref{eq:HJB} is studied without any nondegeneracy assumption on $g$ and a unique G\^ateaux differentiable mild solution to \ref{eq:HJB} is found by an approach based on backward stochastic differential equations (BSDEs). 

The main difficulty, in our case, is due to the presence of unbounded 
terms such as $(I-B)P g$ which forces us to prove extra regularity for the solution to \ref{eq:HJB}.

In this paper we develop the BSDE techniques along the lines initiated by Fuhrman and Tessitore in \cite{FuTe/2004} and Hu and Tessitore \cite{HuTe/2006}.  
In particular, the BSDE corresponding with our problem is given by
\begin{align}\label{eq:BSDE}\tag{BSDE}
           \dd Y(\tau) = \lambda Y(\tau)\dd t- \psi(\bx(\tau;x),Z(\tau))\dd t +Z(\tau) \dd W(\tau), \tau\geq 0
\end{align}
together with a suitable growth condition (which substitutes the 
final condition of the finite horizon case). 
In the above formula, $\bx(\cdot;x)$ is the solution of the equation \eqref{eq:state-eq} starting from $x\in X$ at time $t=0$. 



Our purpose is to prove that the solution of \eqref{eq:BSDE} exists and it is unique for any positive $\lambda$ and that the mild solution of the \eqref{eq:HJB} equation is given in terms of the solution of the \eqref{eq:BSDE}. In particular, if we set $v(x):=Y(0)$ it follows that $v$ is the mild solution of \eqref{eq:HJB}. 
As notice before, in our case it is not enough to prove that $v$ is once (G\^ateaux) differentiable to give sense to equation \ref{eq:HJB}. Indeed the occurrence 
of the term $\nabla v(\cdot) (I-B)Pg$, together with the fact that $P$  maps $H$ into an interpolation space of $D(B)$, forces us to prove that the map $h \mapsto \nabla v(x)(I-B)^\theta Pg[h]$ for suitable $\theta\in (0,1)$ extends to a continuous map on $H$. To do that  we start proving that this extra regularity holds, in a suitable sense, for the state equation \eqref{eq:state-eq} and then that it is conserved if we differentiate in G\^ateaux sense the backward equation with respect to the process $\bx$.   
On the other side, we can prove that 
if the map $h \mapsto \nabla v(x)(I-B)^\theta [h]$ extends to a continuous function on $H$ then 
the processes $t \mapsto v(\bx(t;x))$ and $W$ admit joint quadratic variation in any interval $[t,T]$ and this is given by 
\begin{align*}
      \int_t^T \nabla v(\bx(\tau;x))(I-B)Pg\,\dd \tau.
\end{align*}
This is result is standard and is done by an application of the Malliavin calculus (on a finite time horizon).

We can then come back to the control problem and using the probabilistic representation of the unique mild solution to equation \eqref{eq:HJB} we easily show existence of an optimal feedback law.  Indeed we are able to prove that the so called \emph{fundamental relation} which states that
$\bJ(x,\gamma) \geq v(x)$ and equality holds if and only if the feedback law 
$$
         \gamma(\tau)\in \Gamma(\bx^\gamma(\tau;x),\nabla v(\bx^\gamma(\tau;x)(I-B)P\,g(\bx(\tau;x))))
$$
is verified,
with $\Gamma(x,v)$ the set of minimizers in \eqref{eq:ham} and $\bx^\gamma(\cdot;x)$ the solution of equation \eqref{eq:state-eq} corresponding with the control $\gamma$. We refer to Section \ref{sec:syn-cont} for precise statements and additional results.

The paper is organized as follows: in the next section we give some useful notation and we introduce the main assumptions on the
coefficients of the problem. In Section \ref{sec:anal-set} we reformulate the problem
into a semilinear abstract evolution equation and we study the properties of
leading operator. In Section \ref{sec:exun-abs} we prove the first main result of the paper: we
determine existence and uniqueness of the solution of the reformulated equation \eqref{eq:state-eq}. In Section \ref{sec:exun-orig} we turn back to the original Volterra equation \eqref{eq:Volterra} in order to establish even in this case existence and uniqueness of the solutions. In section \ref{sec:forward} we give some properties of the forward equation corresponding with the reformulated uncontrolled problem. In Section \ref{sec:backward} we study existence, uniqueness and properties of the backward stochastic equation \eqref{eq:BSDE}. In Section \ref{sec:HJB} we proceed with the study of the Hamilton Jacobi Bellman equation and, finally, in section \ref{sec:syn-cont} we employ the results proved in the preceedings sections in order to perform the standard synthesis of the optimal control.

\section{Notations and main assumptions} \label{sec:not-ass}
The norm of an element $x$ of a Banach space $E$ will be denoted by $|x|_E$ or simply
$|x|$ if no confusion is possible. If $F$ is another Banach space, $L(E,F)$ denotes the space 
of bounded linear operators from $E$ to $F$, endowed with the usual operator norm. 

The letters $\Xi,\, H, \, U $ will always denote Hilbert spaces. Scalar product is denoted 
$\langle \,\cdot\, , \, \cdot \, \rangle$, with a subscript to specify the space, if 
necessary. All Hilbert space are assumed to be real and separable. 

By a cylindrical Wiener process with values in a Hilbert space $\Xi$, defined on a probability space $(\Omega,\cF,\PP)$, we mean a family $(W(t))_{t\geq 0}$ of linear mappings from
$\Xi$ to $L^2(\Omega)$, denoted $\xi \mapsto \langle \xi, W(t)\rangle$ such that
\begin{enumerate}
     \item for every $\xi \in \Xi, ( \langle \xi, W(t) \rangle )_{t\geq 0}$ is a real 
     (continuous) Wiener process;
     \item for every $\xi_1,\xi_2 \in \Xi$ and $t\geq 0$, $\bE(\langle \xi_1, W(t) \rangle \langle \xi_2, W(t) \rangle) = \langle \xi_1, \xi_2 \rangle$.
\end{enumerate}

$(\cF_t)_{t\geq 0}$ will denote the natural filtration of $W$, augmented with the family 
of $\P$-null sets. The filtration $(\cF_t)_{t\geq 0}$ satisfies the usual conditions. 
All the concepts of measurability for stochastic processes refer to this filtration. 
By $\cB(\Gamma)$ we mean the Borel $\sigma$-algebra of any topological space
$\Gamma$.

In the sequel we will refer to the following class of stochastic processes with values 
in an Hilbert space $K$:
\begin{enumerate}
      \item  $L^p(\Omega; L^2(0,T;K))$ defines, for $T>0$ and $p\geq 1$,  
        the space of equivalence classes of progressively measurable processes 
        $y: \Omega \times [0,T) \to K$, such that 
        $$
               |y|^p_{L^p(\Omega; L^2(0,T;K))} := \bE \left[ \int_0^T |y(s)|^2_K \dd s\right]^{p/2}
                <\infty.
        $$
        Elements of $L^p(\Omega; L^2(0,T;K))$ are identified up to modification. 
      \item $L^p(\Omega; C([0,T];K))$ defines, for $T>0$ and $p\geq 1$,  
        the space of equivalence classes of progressively measurable processes 
        $y: \Omega \times [0,T) \to K$, with continuous paths in $K$, such that 
        the norm 
        $$
               |y|^p_{L^p(\Omega; C([0,T];K))} := \bE \left[ \sup_{t \in [0,T]} |y(t)|^p_K \right]
        $$
        is finite. Elements of $L^p(\Omega; C([0,T];K))$ are identified up to indistinguishability.
        \end{enumerate}
        
        We also recall notation and basic facts on a class of differentiable maps acting 
        among Banach spaces, particularly suitable for our purposes (we refer the reader to
        Fuhrman and Tessitore \cite{FuTe/2002} or Ladas and Lakshmikantham \cite[Section 1.6]{LaL}
         (1970) for details and ). 
        Let now $X,Y,V$ denote Banach spaces. We say that a mapping $F: X \to V$ 
        belongs to the class $\cG^1(X,V)$ if it is continuous, G\^ateaux differentiable on X,
        and its G\^ateaux derivative $\nabla F: X \to L(X,V)$ is strongly continuous.
        
        The last requirement is equivalent to the fact that for every $h\in X$ the map
        $\nabla F(\cdot)h: X \to V$ is continuous. Note that $\nabla F: X \to L(X,V)$
        is not continuous in general if $L(X,V)$ is endowed with the norm operator 
        topology; clearly, if it happens then $F$ is Fr\'echet differentiable on $X$.
        It can be proved that if $F\in \cG^1(X,V)$ then $(x,h) \mapsto \nabla F(x)h$
        is continuous from $X \times X$ to $V$; if, in addition, $G$ is in $\cG^1(V,Z)$ then
        $G(F)$ is in $\cG^1(X,Z)$ and the chain rule holds: 
        $\nabla(G(F))(x)=\nabla G(F(x))\nabla F(x)$. When $F$ depends on additional 
        arguments, the previous definitions and properties have obvious generalizations.

        Moreover, we assume the following. 
     \medskip
       \begin{hypothesis}\label{hp:a,A,f,g,r,W}
              \begin{enumerate}
                 \item The kernel $a: (0,\infty) \to \R$ is completely monotonic, locally
                  integrable, with $a(0+)=+\infty$. The singularity in $0$ shall satisfy some technical
                 conditions that we make precise in Section \ref{sec:anal-set}. 
                 \item $A: D(A) \subset H \to H$ is a sectorial operator in $H$. Thus $A$ generates an
                 analytic semigroup $e^{tA}$. 
                 \item The function $f:  H \to H$ is measurable and continuously 
                 G\^ateaux differentiable; moreover
                 there exist constants $L>0$ and $C>0$ such that 
                 \begin{align*}
                        & |f(u)-f(v)| \leq L |u-v|, \qquad \ u,v \in H;\\
                        & |f(0)| + \|\nabla_u f(u)\|_{\mathcal{L}(H)} \leq  C, \qquad  \ u\in H
                 \end{align*}
                 \item The mapping $g$ belongs to $L_2(\Xi,H)$ (the space of Hilbert-Schmidt operators from $\Xi$ to $H$);
                 \item The function $r: H \times U \to \Xi$ is Borel measurable and there 
                  exists a positive constant $C>0$ such that 
                 \begin{align*}
                         & |r(u_1,\gamma)-r(u_2,\gamma)| \leq C|u_1-u_2|, \qquad
                         u_1,u_2 \in H, \gamma \in U;\\
                         & |r(u,\gamma)|_\Xi \leq C, \qquad u\in H, \gamma \in U.
                 \end{align*}
                 \item The process $(W(t))_{t\geq 0}$ is a cylindrical Wiener process defined on 
                 a complete probability space $(\Omega,\cF,(\cF_t)_{t\geq 0},\P)$ with values in the
                 Hilbert space $\Xi$.
              \end{enumerate}
        \end{hypothesis}
         \medskip
        The initial condition satisfies a global exponential bound as well as a linear 
         growth bound as $t\to 0$:
         \begin{hypothesis}\label{hp:dato-in}
              \begin{enumerate}
                  \item There exist $M_1 >0$ and $\omega >0$ such that $|u_0(t)|\leq M_1 e^{\omega t}$ for all $t\leq 0$;
                  \item There exist $M_2>0$ and $\tau >0$ such that $|u_0(t)-u_0(0)|\leq M_2 |t|$ for all 
                 $t \in [-\tau,0]$;
               \item $u_0(0)\in H_{\varepsilon}$ for some $\varepsilon \in (0,1/2)$. 
              \end{enumerate}
          \end{hypothesis}
          \medskip
        Concerning the cost functional $\ell$ we make the
        following general assumptions:   
        \begin{hypothesis}\label{hp:l}
            The function $\ell:  H \times U \to \R$ 
               is continuous and bounded.
       \end{hypothesis}
        
        We consider the following notion of weak solution for the Volterra equation \eqref{eq:Volterra}.
        \begin{defn}
              We say that a process $u=(u(t))_{t\geq 0}$ is a weak solution to equation \eqref{eq:Volterra}
             if $u$ is an adapted, $p$-mean integrable (for any $p\geq 1$), continuous 
             $H$-valued predictable process and the identity
              \begin{multline*}
                    \int_{-\infty}^t a(t-s)\langle u(s),\zeta\rangle_H \dd s = 
                     \langle \bar{u},\zeta \rangle_H + \int_0^t \langle u(s),A^\star\zeta\rangle_H  \dd s \\
                      +\int_0^t \langle f(u(s)), \zeta \rangle \dd s + \int_0^t \langle g \, r(u(s),\gamma(s)),\zeta \rangle_H \dd s
                      +  \langle g W(t),\zeta\rangle  
              \end{multline*}
           holds $\P$-a.s. for arbitrary $t\in [0,T]$ and $\zeta \in D(A^\star)$,  with $A^\star$ being the adjoint of the operator $A$
           and 
           \begin{align*}
                   \bar{u}=\int_{-\infty}^0 a(-s) u_0(s) \dd s.
           \end{align*}
        \end{defn}
       
\section{The analytical setting}\label{sec:anal-set}
     A completely monotone kernel $a: (0,\infty) \to \R$ is a continuous, monotone decreasing function, infinitely often derivable, such that
$$
       (-1)^n \frac{\dd^n }{\dd t^n}a(t) \geq 0, \quad t\in (0,\infty), \ n=0,1,2,\dots
$$   
By Bernstein's theorem, $a$ is completely monotone if and only if there exists a positive measure 
   $\nu $ on $[0,\infty)$ such that
  $$
      a(t)=\int_{[0,\infty)} e^{-\kappa t}\nu(\dd \kappa), \quad t>0.
$$
Under the assumption $a\in L^1(0,1)$, it holds that the Laplace transform $\hat{a}$ is well defined and it is given in terms of $\nu$ by
\begin{align*}
      \hat{a}(s)=\int_{[0,\infty)} \frac{1}{s+\kappa}\nu(\dd \kappa).
\end{align*}
We introduce the quantity
$$
      \alpha(a)= \sup\left\{\rho \in \R: \ \int_c^\infty s^{\rho - 2}\frac{1}{\hat{a}(s)}\dd s <\infty\right\}
$$
and we make the following assumption:
\begin{hypothesis}
     $\alpha(a)>1/2$.
\end{hypothesis}

\begin{rem}
It is known from the theory of deterministic Volterra equations that the singularity 
of $a$ helps smoothing the solution. We notice that $\alpha(a)$ is independent on the choice of $c>0$ and this quantity 
describes the behavior of the kernel near $0$; by this way we ensure that smoothing is sufficient to keep
the stochastic term tractable. 
\end{rem}

It is known that we can associate to any completely monotone kernel
$a$, by means of Bernstein's Theorem \cite[pag. 90]{pruss}, a measure
$\nu$ on $[0,+ \infty)$ such that
\begin{equation}
  \label{eq:Bernstein}
  a(t) = \int_{[0,+ \infty)} e^{-\kappa t}\, \nu({\rm d}\kappa).
\end{equation}
From the required singularity of $a$ at $0+$ we obtain that \(\nu([0,+
\infty))=a(0+)=+ \infty\) while for $s>0$ the Laplace transform $\hat
a$ of $a$ verifies
\begin{align*}
  \hat a(s) = \int_{[0,+ \infty)} \frac 1{s+\kappa} \, \nu({\rm
    d}\kappa) < + \infty.
\end{align*}

Under the assumption of complete monotonicity of the kernel, a
semigroup approach to a type of abstract integro-differential
equations encountered in linear viscoelasticity was introduced in
\cite{desch/miller/1988} and extended  to the case of Hilbert space valued equations in \cite{bonaccorsi/desch}. 

We will see that  this approach allow to treat the case of semilinear, stochastic integral equations; we start for simplicity with the equation
\begin{equation}
  \label{eq:ev-f}
  \begin{aligned}
    \frac{d}{dt} \int_{-\infty}^t a(t-s)u(s) \, {\rm d}s &= Au(t) +
    f(t), \qquad& t \in [0,T]
    \\
    u(t) &= u_0(t), \qquad& t \le 0,
  \end{aligned}
\end{equation}
where $f$ belongs to $L^1(0,T;X)$.
Moreover, we introduce the following identity, which comes from
Bernstein's theorem
\begin{align*}
  \int_{-\infty}^t a(t-s)u(s)\, {\rm d}s = \int_{-\infty}^t
  \int_{[0,+ \infty)} e^{-\kappa(t-s)} \, \nu({\rm d}\kappa)\, u(s) {\rm
    d}s = \int_{[0,+ \infty)} \bx(t,\kappa)\, \nu({\rm d}\kappa);
\end{align*}
here $\bx(t,\kappa)$ is function given by
\begin{equation}
  \label{eq:intro_v_defined}
  \bx(t,\kappa) = \int_{-\infty}^t e^{-\kappa(t-s)}u(s)\, {\rm d}s.
\end{equation}
Formal differentiation yields
\begin{equation}
  \label{eq:intro_v_diff_equation}
  \frac{\partial}{\partial t} \bx(t,\kappa) = -\kappa \bx(t,\kappa) +
  u(t),
\end{equation}
while the integral equation (\ref{eq:ev-f}) can be rewritten
\begin{equation}
  \label{eq:intro_rewritten}
  \int_{[0,+ \infty)} (-\kappa \bx(t,\kappa)+u(t))\, \nu({\rm d}\kappa) =
  A u(t) + f(t).
\end{equation}
Now, the idea is to use equation \eqref{eq:intro_v_diff_equation} as
the state equation, with $B \bx= -\kappa \bx(\kappa) + u$, while
\eqref{eq:intro_rewritten} enters in the definition of the domain of
$B$.

In our setting, the function \(\bx(t,\cdot)\) will be considered the
state of the system, contained in the state space $X$
that consists of all Borel measurable functions $
\by: [0,+ \infty) \to H$ such that the seminorm
\begin{align*}
  \Vert \by \Vert_X^2 := \int_{[0,+ \infty)} (\kappa + 1) \vert 
\by(\kappa)\vert_H^2 \, \nu({\rm d}\kappa)
\end{align*}
is finite. We shall identify the classes $\by$ with respect to
equality almost everywhere in $\nu$.

Let us consider the initial condition. We introduce the space 
$$
   \tilde X_0:=\left\{ u: \R_{-} \to H: \ exists \ M>0 \ and \ \omega>0 \ s.t. \ |u(t)|<Me^{\omega t}, \ t\leq 0 \right\}$$
and we endow it with a  positive inner product
\begin{align*}
  \langle u,v \rangle_{\tilde X} = \int \int [a(t+s) - a'(t+s)]
  \langle u(-s), v(-t) \rangle_H \, {\rm d}s \, {\rm d}t;
\end{align*}
then, setting $\tilde N_0 = \{ u \in \tilde X_0\ :\ \langle u,u
\rangle_{\tilde X} = 0\}$, $\langle \cdot,\cdot \rangle_{\tilde X}$ is
a scalar product on $\tilde X_0 / \tilde N_0$; we define $\tilde X$
the completition of this space with respect to $\langle \cdot,\cdot
\rangle_{\tilde X}$. We let the operator $Q: \tilde X \to X$ be given
by
\begin{equation}
  \label{eq:intro_initial}
  \bx(0,\kappa)= Q u_0(\kappa) = \int_{-\infty}^0 e^{\kappa s}u_0(s)\, {\rm d}s.
\end{equation}

It has been proved in \cite[Proposition 2.5]{bonaccorsi/desch} that the
operator $Q$ is an isometric isomorphism between $\tilde X$ and $X$.
This operator maps the initial value of the stochastic Volterra
equation in the initial value of the abstract state
equation. Different initial conditions of the Volterra equation
generate different initial conditions of the state equation.

Hypothesis \ref{hp:dato-in} is necessary in order
to have a greater regularity on the inial value of the state
equation. In fact in this case \cite[Proposition 2.20]{bonaccorsi/desch} shows that $Q u_0$ belongs to $X_{\eta}$ for
$\eta \in (0,\frac12)$.

\begin{rem}
  We stress that under our assumptions we are able to treat, for
  instance, initial conditions for the Volterra equation of the
  following form
  \begin{align*}
    u_0(t) =
    \begin{cases}
      0, & ( - \infty, -\delta); \\
      \bar{u} & [-\delta,0]
    \end{cases}
  \end{align*}
    provided $\bar{u}$ has a suitable regularity.
\end{rem}

\smallskip

We quote from \cite{bonaccorsi/desch} the main result concerning the
state space setting for stochastic Volterra equations in infinite
dimensions.

\begin{thm}[State space setting]
  \label{t:state space setting}
  Let \(A\), \(a\), \(\alpha(a)\), \(W\) be given above; choose
  numbers \(\eta \in (0,1)\), \(\theta \in (0,1)\) such that
  \begin{align}\label{eq:eta-theta}
    \eta > \frac 12\, (1-\alpha(a)), \quad \theta < \frac 12\,
    (1+\alpha(a)), \quad \theta-\eta>\frac 12.
  \end{align}
  Then there exist
  \hfill\begin{itemize}
  \item[1)] a separable Hilbert space \(X\) and an isometric
    isomorphism \(Q: \tilde X \to X\),
  \item[2)] a densely defined sectorial operator \(B:D(B) \subset X
    \to X\) generating an analytic semigroup \(e^{tB}\) with growth
    bound $\omega_0$,
  \item[3)] its real interpolation spaces \(X_{\rho} =
    (X,D(B))_{(\rho,2)}\) with their norms \(\Vert \cdot
    \Vert_{\rho}\),
  \item[4)] linear operators \(P: H \to X_{\theta}\), \(J:X_{\eta} \to
    H\)
  \end{itemize}
  such that the following holds:

   For each \(x_0 \in X\), the problem (\ref{eq:ev-f}) is
  equivalent to the evolution equation
    \begin{equation}
      \label{pb:evo-state-f}
      \begin{aligned}
        \bx'(t) &= B \bx(t) + (I-B)P f(t)
        \\
        \bx(0) &= x_0
      \end{aligned}
    \end{equation}
    in the sense that if \(u_0\in \tilde X_0\) and \(\bx(t;x_0)\) is the
    weak solution to Problem~\eqref{pb:evo-state-f} with \(x_0 = Qu_0\),
    then \(u(t;u_0) = J\bx(t;x_0)\) is the unique weak solution to
    Problem~\eqref{eq:ev-f}.
\end{thm}



It is remarkable that $B$ generates an analytic semigroup, since in
this case we have at our disposal a powerful theory of optimal
regularity results. In particular, besides the interpolation spaces
$X_\theta$ introduced in Theorem \ref{t:state space setting}, we may
construct the extrapolation space $X_{-1}$, i.e., a Sobolev space of
negative order associated to $e^{t B}$. 

Under some additional condition it is possible to prove exponential stability
 of the semigroup $e^{tB}$.
 As we will see, this is especially needed in the study of the optimal control problem with infinite horizon.
In particular we impose the following
\begin{hypothesis}\label{hp:expdec}
   There exists $\sigma >0$ such that the function $e^{\sigma t}a(t)$ is completely 
   monotonic.
\end{hypothesis}

\noindent Consequently, we obtain that $B$ is of negative type.
\begin{prop}\label{prop:expdec}
    The real parts of the spectrum of $B$ are bounded by some $\omega_0<0$. Consequently the semigroup $e^{tB}$ decays exponentially.  
\end{prop}
\begin{proof}
    We proceed as in Bonaccorsi and Desch \cite{bonaccorsi/desch}. All we need to show is that $0$ is in the resolvent set of $B$. Once this
is proved, the spectral bound is negative since the spectrum is confined to a
sector. From this exponential decay follows, since analytic semigroups have
spectrum determined growth.
Now we notice that $\hat{a}(s)$ exists in the set $\mathbb{C}\setminus (-\infty,-\sigma]$
and $0$ is in the resolvent set of $A$, by assumption. 
Moreover, we can give the explicit expression of 
$R(s,B)$, which is
\begin{align*}
   [R(s,B)\bx](\kappa)= \frac{1}{s+\kappa} \left[ \bx(\kappa)+ R(s\hat{a}(s)) \int_{[0,\infty)} \frac{\kappa}{s+\kappa}\bx(\kappa) \nu (\dd \kappa)\right]
\end{align*}
(the calculation is straightforward and can be found in \cite[Lemma 3.5]{bonaccorsi/desch}).
Then it is easily seen that $R(s,B)$ can be extended to a neighborhood of $0$. 
\end{proof}

The semigroup $e^{t B}$ extends to $X_{-1}$ and the generator of this
extension, that we denote $B_{-1}$, is the unique continuous extension
of $B$ to an isometry between $X$ and $X_{-1}$. See for instance
\cite[Definition 5.4]{Engel2000} for further details.

\begin{rem}
  \label{re:2.**}
  In the sequel, we shall always denote the operator with the letter
  $B$, even in case where formally $B_{-1}$ should be used
  instead. This should cause no confusion, due to the similarity of
  the operators.
\end{rem}

\begin{rem}
   \label{re:riscrittura}
   If we apply to problem \eqref{eq:Volterra} the machinery introduced above, we obtain, on the space
 $X$, the stochastic Cauchy problem
       \begin{align}\label{eq:state-eq-2}
   \begin{cases}
    \dd \bx(t)= B \bx(t)\dd t+(I-B)Pf(J\bx(t))\dd t \\
         \qquad \qquad 
           +g\,r(J \bx(t),\gamma(t))\dd t+ (I-B)P g\,\dd W(t)\\
     \bx(0)=x.
   \end{cases}
   \end{align}
    for $0\leq t\leq T$ and initial condition $x\in X_\eta$. 
The above expression is only formal since the coefficients do not belong to the state space; however, we can give a meaning to the mild form of the equation and we can also prove existence and uniqueness of the solution. Moreover,
  given the solution $\bx$ of \eqref{eq:state-eq-2}, we obtain the solution of 
  the original Volterra equation by setting $v(t)=J\bx(t)$. This facts will be the object of the next sections.
\end{rem}

\section{The state equation: existence and uniqueness}\label{sec:exun-abs}
       In this section, motivated by the construction in Section \ref{sec:anal-set}, we shall establish 
      existence and uniqueness result for the following stochastic uncontrolled
Cauchy problem on the space $X$ defined in Section \ref{sec:anal-set}:
      \begin{align}\label{eq:unc}
   \begin{cases}
    \dd \bx(t)= B \bx(t)\dd t+(I-B)Pf(J\bx(t))\dd t \\
         \qquad \qquad 
           +(I-B)P g \, r(J\bx(t),\gamma(t))\dd t + (I-B)P g\,\dd W(t)\\
     \bx(0)=x.
   \end{cases}
   \end{align}
    for $0 \leq  t \leq T$ and initial condition $x\in X_\eta$. The above expression is only formal since the coefficients do not belong to the state space; however, we can give a meaning to the mild form of the equation: 
\begin{defn} 
We say that a continuous, $X$-valued, predictable process $\bx=(\bx(t))_{t\geq 0}$ is a (mild) solution of the state equation \eqref{eq:state-eq-2} if $\P$-a.s.,
   \begin{multline*}
        \bx(t)=e^{(t-s)B}x+\int_s^t e^{(t-\sigma)B}(I-B)Pf(J\bx(\sigma))\dd \sigma\\
          \quad +\int_0^t e^{(t-\sigma)B}(I-B)P \, r(J\bx(\sigma),\gamma(\sigma))\dd \sigma+ \int_s^t e^{(t-\sigma)B}(I-B)P g\,\dd W(\sigma).
   \end{multline*}
\end{defn}
   Let us state the main existence result for the solution of equation \eqref{eq:state-eq-2}.
   
   \begin{thm}\label{thm:exuni}
        Under Hypotheses \ref{hp:a,A,f,g,r,W}, \ref{hp:dato-in}, for an arbitrary predictable process $\gamma$ with values in $\cU$, for every $0\leq  t\leq T$ and $x\in X_\eta$, there exists a unique adapted process $\bx \in L^p_\cF(\Omega,C([0,T];X_\eta))$ solution of \eqref{eq:state-eq-2}. 
Moreover,
 the estimate
   \begin{align}\label{eq:exuni}
       \bE \sup_{t\in[0,T]} ||\bx(t)||^p_\eta \leq C (1+||x||_\eta^p)
\end{align}
holds for some positive constant $C$ depending on $T$ and the parameters of the problem.
   \end{thm}

\begin{proof}
       The proof of the above theorem proceeds, basically, on the same lines as the proof of
       Theorem 
        $3.2$ in Bonaccorsi and Mastrogiacomo \cite{BoMa-JEE} (2009). 
      First, we define a mapping $\cK$ from $L^p(\Omega;C([0,T];X_\eta))$ to itself by the formula
      \begin{align}\label{eq:kappa}
           \cK(\bx)(t):=e^{tB}x+ \Lambda(\bx)(t)+\Delta(\bx)(t)+\Gamma(t),
      \end{align}
      where the second, third and last term in the right side of \eqref{eq:kappa} are given by
     \begin{align}
      &\Lambda(\bx)(t)=\int_0^t e^{(t-\tau)B}(I-B)Pf(J\bx(\tau))\dd \tau \label{eq:lambda} \\
      &\Delta(\bx)(t)=\int_0^t e^{(t-\tau)B}(I-B)P g \, r(J \bx(\tau),\gamma(\tau)) \dd \tau\\
      &\Gamma(t)= \int_0^t e^{(t-\tau)B}(I-B)P g\,\dd W(\tau) \label{eq:gamma}
      \end{align}
      Then, we will prove that the mapping $\cK$ is a contraction 
      on $L^p(\Omega;C([0,T];X_\eta))$ with respect to the equivalent norm
      $$
             \tn \bx\tn^p_\eta := \bE \sup_{t\in [0,T]} e^{-\beta p t}||\bx(t)||^p_\eta,
       $$
      where $\beta>0$ will be chosen later. For simplicity we write $\Lambda(t)$ instead of $\Lambda(\bx)(t)$.  

     Our first step is to prove that $\Gamma, \Delta$ and $\Lambda$ are well-defined
     mappings on the space $L^p(\Omega;C([0,T];X_\eta))$ and to give estimates on their norm.
     We choose $\delta$ small enough such that $ 1+\eta-\theta +1/p<\delta<\!
     <1/2$ and define
      \begin{align*}
            y_\eta(\tau) := \int_0^t (t-\sigma)^{-\delta}e^{(t-\sigma)B}(I-B)^{\theta}P\,g \dd W(\sigma). 
      \end{align*}
      Since  the semigroup $e^{tB}$ is analytic, $P$ maps $H$ into $X_{\theta}$ and $g\in L_2(\Xi,H)$, an application of Lemma 7.2 in \cite{LaL}, yields:
     $\bE\int_0^T |y_\eta(\sigma)|^p <\infty$. In particular 
     $y\in L^p([0,T];X)$, $\PP$-a.s. Moreover, if we set
       $$
     ( R_\delta\phi)(t) =\int_0^t (t-\sigma)^{\delta-1} e^{(t-\sigma)B} (I-B)^{1+\eta-\theta}\phi(\sigma) \dd \sigma,
     $$
       then in \cite[Proposition A.1.1.]{DPZ92} it is proved that $R_\delta$ is a bounded linear operator 
     from $L^p([0,T];X)$ into $C([0,T];X)$. Finally, by stochastic Fubini-Tonelli 
     Theorem we can rewrite:
     \begin{align*}
          (R_\delta y_\eta)(t)&= \int_0^t \int_0^\tau (t-\tau)^{\delta -1}(\tau-\sigma)^{\delta}\\
        &\qquad \qquad \qquad (I-B)^{\eta+1} e^{(t-\sigma)B}P\, g \dd W(\sigma) \ \dd \tau \\
       &=\left(\int_0^1 (1-\tau)^{\delta-1}\tau^{-\delta}\dd \tau\right)
       (I-B)^{\eta} \Gamma(t)
      \end{align*}
     and conclude that $\Gamma(t)\in L^p(\Omega, C([0,T];X_\eta))$. 

     In a similar (and easier) way it is possible to show that $\Lambda(\cdot,t)$ and $\Delta(\cdot,t)$ belong
     to $L^p(\Omega, C([0,T];X_\eta))$. 
     hence, we conclude that $\cK$ maps $L^p(\Omega;C([0,T];X_\eta))$
     into itself.

    Now we claim that $\cK$ is a contraction in $L^p(\Omega, C([0,T];X_\eta))$. In fact,
    by straightforward estimates we can write
    \begin{align*}
        \tn \Lambda (\bx)(t) -\Lambda(\by)(t)\tn^p_\eta \leq C_{L,T} \beta^{1/2+\delta-(\theta-\eta)} 
     \tn \bx -\by \tn_\eta^p.
    \end{align*}
     Therefore $\cK$ is Lipschitz continuous from $L^p(\Omega,C([0,T];X_\eta))$ into itself;
further, we can find $\beta$ large enough such that $C_{L,T}(2\beta)^{1/2+\delta+\eta-\theta}<1$.
 Hence $\cK$ becomes a contraction on the time interval $[0,T]$ and by a classical 
     fixed point argument we get that there exists  a unique solution of the equation
     \eqref{eq:state-eq-2} on $[0,T]$.
\end{proof}

   \begin{rem}
       In the following it will be also useful to consider the uncontrolled version of 
     equation \eqref{eq:state-eq-2}, namely:
     \begin{align}\label{eq:un-state-eq}
    \begin{cases}
    \dd \bx(t)= B \bx(t)\dd t+(I-B)Pf(J\bx(t))\dd t+ (I-B)P g\,\dd W(t)\\
     \bx(0)=x.
   \end{cases}
    \end{align}
    We will refer to \eqref{eq:un-state-eq} as the forward equation.
      We then notice that existence and uniqueness for the above equation can be treated in an identical 
     way as in the proof of Theorem \ref{thm:exuni}. 
   \end{rem}

\section{The controlled stochastic Volterra equation}\label{sec:exun-orig}
In this section we prove existence and uniqueness of solution for the original Volterra equation
\eqref{eq:Volterra}. As a preliminary step for the sequel, 
 we state two results of existence and uniqueness for (a special case of) the original Volterra
equation. The proofs can be found in \cite[Section 2]{BCM11}.


The first result concerns with the solution to the linear deterministic Volterra equation.
\begin{prop} 
  The linear equation
  \begin{equation}
    \label{eq:ev-lin}
    \begin{aligned}
      \frac{d}{dt} \int_{-\infty}^t a(t-s)u(s) \, {\rm d}s &= A
      u(t), \qquad t \in [0,T]
      \\
      u(t) &= 0, \qquad t \le 0.
    \end{aligned}
  \end{equation}
  has a unique weak solution $u \equiv 0$.
\end{prop}

%


\smallskip

The second result deals with existence and uniqueness of the Stochastic Volterra
equation with non-homogeneous terms. The result comes from a generalization of \cite[Theorem 3.7]{bonaccorsi/desch}, where the case $f(t) \equiv 0$
is treated.  

\smallskip

\begin{prop}
  \label{te:4.2}
  In our assumptions, let $x_0 \in X_\eta$ for some
  $\frac{1-\alpha(a)}{2} < \eta < \frac{1}{2} \alpha(a)$.  Given the
  process
  \begin{align}
    \label{eq:v-lineare}
    \bx(t) = e^{t B}x_0 + \int_0^t e^{(t-s)B} (I-B) P f(s) \, {\rm d}s +
    \int_0^t e^{(t-s)B} (I-B) P g  \, {\rm d}{W(s)}
  \end{align}
  we define the process
  \begin{equation}
    \label{eq:definition_of_u}
    u(t) = \begin{cases}
      J \bx(t), & t \ge 0, \\
      u_0(t), & t \le 0.
    \end{cases}
  \end{equation}
  Then $u(t)$ is a weak solution to problem
  \begin{equation}
    \label{eq:u-lineare}
    \begin{aligned}
      \frac{d}{dt} \int_{-\infty}^t a(t-s)u(s) \, {\rm d}s &= A u(t) +
      f(t)+ g\dot W(t), \qquad t \in [0,T]
      \\
      u(t) &= u_0(t), \qquad t \le 0.
    \end{aligned}
  \end{equation}
\end{prop}

After the preparatory results stated above, here we
prove that main result of existence and uniqueness of solutions of the
original controlled Volterra equation
\eqref{eq:Volterra}. 

\begin{thm} 
  \label{sol-Vol-contr}
 Assume Hypotheses \ref{hp:a,A,f,g,r,W} and \ref{hp:dato-in}. Let $\gamma$ be an admissible control and $\bx$ be the solution to problem
  (\ref{eq:state-eq}) (associated with $\gamma$) whose existence is proved in Theorem
  \ref{thm:exuni}.
  Then the process
  \begin{equation}
    \label{eq:u}
    u(t) =
    \begin{cases}
      u_0(t), & t \le 0
      \\
      J \bx(t), & t \in [0,T]
    \end{cases}
  \end{equation}
  is the unique solution of the stochastic Volterra equation
  \begin{equation}
    \label{eq:ev-sec4}
    \begin{aligned}
      \frac{d}{dt} \int_{-\infty}^t a(t-s)u(s) \, {\rm d}s &= A u(t) +
      f(u(t)) + g\, \left[r(u(t),\gamma(t)) + \dot W(t)\right],
      \qquad t \in [0,T]
      \\
      u(t) &= u_0(t), \qquad t \le 0.
    \end{aligned}
  \end{equation}
\end{thm}
  
\begin{proof}
  We propose to fulfill the following steps: first, we prove that the affine equation
    \begin{equation}
      \label{eq:ev-lin-2}
      \begin{aligned}
        \frac{d}{dt} \int_{-\infty}^t a(t-s)u(s) \, {\rm d}s &= A u(t)
        + f(\tilde{u} (t)) + g \,\left[r(\gamma (t), \tilde u(t)) + \dot
          W(t) \right], \qquad t \in [0,T]
        \\
        u(t) &= u_0(t), \qquad t \le 0.
      \end{aligned}
    \end{equation}
    defines a contraction mapping $\cQ: \tilde u \mapsto u$ on the
    space $L^2_\cF(\Omega;C([0,T];H))$. Therefore, equation
    (\ref{eq:ev-lin-2}) admits a unique solution.
  
   Then we show that the process $u$ defined in (\ref{eq:u})
    satisfies equation (\ref{eq:ev-lin-2}). Accordingly, by the
    uniqueness of the solution, the thesis of the theorem follows.

  \textit{First step.}
We proceed to define the mapping
\begin{align*}
  \cQ: L^p(\Omega;C([0,T];H)) \to L^p(\Omega;C([0,T];H))
\end{align*}
where $\cQ(\tilde u) = u$ is the solution of the problem
\begin{equation}
  \label{eq:ev-cL}
  \begin{aligned}
    \frac{d}{dt} \int_{-\infty}^t a(t-s)u(s) \, {\rm d}s &= A u(t) +
   f(\tilde{u}(t))+  g\, \left[r(\tilde u(t),\gamma(t)) +  \dot W(t)\right],
    \qquad t \in [0,T]
    \\
    u(t) &= u_0(t), \qquad t \le 0.
  \end{aligned}
\end{equation}


Let $\tilde{u}_1$ and $\tilde{u}_2$ be two processes belonging to $L^p(\Omega;C([0, T ];H))$ and take $u_1 =\bQ(\tilde{u}_1)$
and $u_2 =\bQ(\tilde{u}_2)$.
  It follows from the uniqueness of the solution, proved in Proposition \eqref{eq:ev-lin},
  that the solution $u_i(t)$ ($i=1,2$) has the representation
  \begin{align*}
    u_i(t) =
    \begin{cases}
      J v_i(t), & t \in [0,T]
      \\
      u_0(t), & t \le 0
    \end{cases}
  \end{align*}
  where
  \begin{multline*}
    v_i(t) = e^{t B}v_0 + \int_0^t e^{(t-s)B} (I-B) P g\, r( \tilde u_i(s), \gamma(s)) \, {\rm d}s
    \\
    + \int_0^t e^{(t-s)B} (I-B) P f( \tilde u_i(s)) \, {\rm d}s+ \int_0^t e^{(t-s)B} (I-B) P g \, {\rm
      d}{W(s)}.
  \end{multline*}
  In particular,
  \begin{align*}
    U(t) = u_1(t) - u_2(t) =
    \begin{cases}
      J(v_1(t) - v_2(t)), & t \in [0,T]
      \\
      0, & t \le 0;
    \end{cases}
  \end{align*}
  then
  \begin{align*}
    \bE \sup_{t \in [0,T]} e^{-\beta p t} |U(t)|^p \le
    \|J\|^p_{L(X_\eta,H)} \bE \sup_{t \in [0,T]} e^{-\beta p t} \|v_1(t)
    - v_2(t)\|^p_\eta.
  \end{align*}
  Now we notice that the quantity on the right hand side can be treated as in Theorem
  \ref{thm:exuni} and the claim follows.

$ $

\textit{Second step}

It follows from the previous step that there exists at most a 
unique solution $u$ of problem (\ref{eq:ev-lin-2}); hence it only remains to prove the representation
formula (\ref{eq:u}) for $u$.

Let $\tilde f(t) = f( J \bx(t)) + g r( J \bx(t), \gamma(t))$; it is a consequence of Proposition
\ref{te:4.2} that $u$, defined in (\ref{eq:u}), is a weak solution
of the problem
\begin{equation}
  \label{eq:III-1}
  \begin{aligned}
    \frac{d}{dt} \int_{-\infty}^t a(t-s)u(s) \, {\rm d}s &= A u(t) +
    \tilde f(t) +  g  \dot W(t),
    \qquad t \in [0,T]
    \\
    u(t) &= u_0(t), \qquad t \le 0,
  \end{aligned}
\end{equation}
and the definition of $\tilde f$ implies that $u$ is a
weak solution of
\begin{equation}
  \label{eq:III-2}
  \begin{aligned}
    \frac{d}{dt} \int_{-\infty}^t a(t-s)u(s) \, {\rm d}s &= A u(t) +
    f(J \bx(t)) + g \left[ r( J \bx(t), \gamma(t)) + \dot W(t)
    \right], \ \ t \in [0,T]
    \\
    u(t) &= u_0(t), \qquad t \le 0,
  \end{aligned}
\end{equation}
that is problem (\ref{eq:ev-lin-2}).
\end{proof}

\section{The forward equation}\label{sec:forward}
%
%
%
In this section we state some properties of the forward equation \eqref{eq:state-eq-2} corresponding with our problem. 

In the following, $\bx$ will denote the solution of the uncontrolled equation \eqref{eq:state-eq-2}.
Under the assumption introduced in the previous sections we can give the regular dependence of $\bx$ on the initial condition $x$.
This result will be used later in order to characterize the solution of the stationary HJB in terms of the solution of a suitable stochastic backward differential equation and, consequently, to characterize the optimal control for our problem.  
\begin{prop}\label{prop:dif-state-eq}
    Under Hypotheses \ref{hp:a,A,f,g,r,W} and \ref{hp:expdec}, for any $p\geq 1$ the following holds. 
    \begin{enumerate}
        \item \label{it:cont} The map $x \mapsto \bx(t;x)$ defined on $ X_\eta$ and
        with values in $L^p(\Omega,C([0,T];X_\eta))$ is continuous.
        \item\label{it:eq} The map $x \mapsto \bx(t;x)$ has, at every point
         $x\in X_\eta$, a G\^ateaux derivative $\nabla_x \bx(\cdot;x)$. The map
        $(x,h) \mapsto \nabla_x \bx(\cdot;x)[h]$ is a continuous map
        from $ X_\eta \times X_\eta \to L^p(\Omega,C([0,T];X_\eta))$ and, 
       for every $h\in X_\eta$, the following equation holds $\PP$-a.s.:
        \begin{equation}\label{eq:diffx}
          \begin{aligned}
                \nabla_x \bx(t;x) [h]= e^{tB}h + \int_0^t e^{\tau B} (I-B)P\, 
                 \nabla_u f(
                J \bx(\tau;x)) J \nabla_x \bx(\tau;x) [h]\dd \tau.
\end{aligned}        
\end{equation}
       Finally, $\PP$-a.s., we have 
       \begin{align}\label{eq:bound-grad-x}
       |\nabla_x \bx(t;x) [h]| \leq C|h|,
       \end{align} for all $t >0$ and some $C>0$.
    \end{enumerate}
\end{prop}
\begin{proof}
     Points \ref{it:cont} and \ref{it:eq} are proved, for instance, in  \cite[Proposition 6.2]{CoMa11}.
     To prove the \eqref{eq:bound-grad-x} we simply notice that 
     \begin{multline*}
        \left\|\int_0^t e^{\tau B} (I-B)P\, 
                 \nabla_u f(
                J \bx(\tau;x)) J \nabla_x \bx(\tau;x) [h]\dd \tau
                \right\|_\eta \\
                \leq |P|_{L(H,X_\theta)} \|\nabla_u f\|_{L(H,H)} |J|_{L(X_\eta,H)} \int_0^t e^{-\omega s} s^{-1-\eta+\theta}  |\nabla_x \bx(s;x) [h]| \dd s.
     \end{multline*} 
     By application of Gronwall's lemma we thus obtain
     \begin{align*}
       | \nabla_x \bx(t;x) [h]| \leq C  |h|, \quad t\geq 0, \ \ \P -a.s. 
     \end{align*}
     where $C$ is a positive constant independent of $t$ and $x$.  
\end{proof}

As we will see later, an important point in order to study the HJB equation corresponding with our problem is that of extending the map $h \mapsto \nabla_x \bx(t;x)(I-B)^{1-\theta}Pg[h]$ - a priori defined on $X_{1-\theta}$ - to the
whole space $X$. This result is stated below. 
\begin{prop}\label{prop:diff-state-eq-gen}
    Under assumptions \ref{hp:a,A,f,g,r,W} and \ref{hp:expdec} the map $h\mapsto \nabla_x \bx(t;x) (I-B)^{1-\theta} Pg [h]$ - a priori defined on $X_{1-\theta}$ - can be estended to the whole space $X$ and 
    it is continuous from $[0,T] \times X_\eta \times X $ to $ L^\infty(\Omega;C([0,T];X_\eta))$ for any $T>0$. Finally, there exists a constant
    \begin{align}\label{eq:grad-x-gen}
        | \nabla_x \bx(t;x) (I-B)^{1-\theta} Pg [h]| \leq C |h|, 
        \end{align}
        for all $t\geq 0, x\in X_\eta, h\in X$, with $C$ indenpendent of $t$ and $x$.
    \end{prop}

    \begin{proof}
         We proceed by proving that the norm of $\nabla_x \bx (t;x) (I-B)^{1-\theta}Pg$ in the space of linear operators on $X$ is finite. 
         In fact, taking into account Proposition \ref{prop:dif-state-eq} we see that, for any $h\in X_{1-\theta}$, 
         the process $\nabla_x \bx(t;x) (I-B)^{1-\theta} Pg[h]$ satisfies the following equation:
          \begin{multline*}
                 \nabla_x \bx(t;x) (I-B)^{1-\theta}[h]= e^{tB}(I-B)^{1-\theta}h\\ + \int_0^t e^{(t-\tau) B} (I-B)P\, 
                 \nabla_u f(
                J \bx(\tau;x)) J \nabla_x \bx(\tau;x)(I-B)^{1-\theta} [h]\dd \tau.
  \end{multline*} 
      Hence, recalling that $\nabla_x f$ is bounded and Proposition \ref{prop:expdec}, we can estimate $|\nabla_x \bx(t;x) (I-B)^{1-\theta} Pg[h]|$ as follows:
      \begin{multline*}
         |\nabla_x \bx(t;x) (I-B)^{1-\theta} Pg[h]| \leq  
         e^{-\omega t}t^{\theta-1}  \\
         + C_{f} \|J\| \int_0^t |\nabla_x \bx(t;x) (I-B)^{1-\theta} Pg[h]|e^{-(t-\tau)\omega}(\tau-t)^{\theta-1} |h| \dd \tau 
      \end{multline*}
      Now the bound \eqref{eq:grad-x-gen} follows by an easy application of Gronwall's lemma, while for continuity, we can refer to \cite[Proposition 6.2]{CoMa11}.
    \end{proof}

\section{The backward equation on an infinite horizon}\label{sec:backward}

In this section we consider the backward stochastic differential equation in the unknown
$(Y,Z)$:
\begin{align}\label{eq:bsde2}
       Y(\tau)=Y(T) + \int_\tau^T( \lambda Y(\sigma))-\psi(\bx(\sigma;x),Z(\sigma))\dd \sigma+\int_\tau^T Z(\sigma)\dd W(\sigma), 
\end{align}
where $0\leq \tau \leq T<\infty$, $\lambda >0$, $\bx(\cdot;x)$ is the solution of the uncontrolled equation \eqref{eq:unc} 
and $\psi$ is the Hamiltonian function relative to the control problem described in
Section \ref{sec:intro}. More precisely, for $x\in X_\eta,z\in \Xi^\star$ we have
\begin{align}\label{eq:ham2}
      \psi(x,z):= \inf\left\{ \ell(Jx,\gamma)+z\cdot r(J x, \gamma): \ \gamma \in \cU\right\}.
\end{align} 
We require the following assumption on $\psi$:
\begin{hypothesis}\label{hp:grad-psi}
\begin{enumerate}
  \item $\psi$ is uniformly Lipschitz continuous in $z$, with Lipschitz constant $K$, that is:
  \begin{align*}
      |\psi(x,z_1)-\psi(x,z_2)| \leq K \| z_1-z_2\|_{\Xi^*}.
  \end{align*}
  \item $\sup_{x\in X} |\psi(x,0)| :=M <\infty$. 
   \item   The map $\psi$ is G\^ateaux differentiable on $X_\eta \times \Xi^\star$ and the maps 
      $(x,h,z) \mapsto \nabla_x \psi(x,z)[h]$ and $(x,z,\zeta) \mapsto \nabla_z \psi(x,z)[\zeta]$ are continuous on $X_\eta\times X \times \Xi^\star$ and $X_\eta\times \Xi^\star \times \Xi^\star$ respectively. 
      \end{enumerate}
\end{hypothesis}


The existence and uniqueness of solution to \eqref{eq:bsde2}
under Hypothesis \ref{hp:grad-psi} was first studied (even though for more general coefficients $\psi$) by Briand and Hu in \cite{BrHu08} and successively by Royer in \cite{Ro04}. Their result in valid when $W$ is a finite dimensional Wiener process but the extension to the case in which $W$ is
a Hilbert-valued Wiener process is immediate.

In our context the result reads as follows.
%
\begin{prop}\label{prop:bsde}
     Assume Hypothesis \ref{hp:l} and \ref{hp:grad-psi}. 
    Then we have:
   \begin{enumerate}
        \item For any $x\in X_\eta$, there exists a solution $(Y,Z)$ to BSDE \eqref{eq:bsde2} such that $Y$ is a continuous process bounded by $\frac{M}{\lambda}$, and $Z \in L^2(\Omega;L^2(0,\infty;\Xi))$ with 
   $\bE \int_0^\infty e^{-2\lambda s }|Z(s)|^2\dd s <\infty$. Moreover, the solution is unique in the class of processes $(Y,Z)$ such that $Y$ is continuous and uniformly bounded, and
      $Z$ belongs to $L^2_{loc}(\Omega,L^2(0,\infty;\Xi^\star))$. In the following we will denote such a solution 
       by $Y(\cdot;x)$ and $Z(\cdot;x)$.
   \item Denoting by $(Y(\cdot;x,n),Z(\cdot;x,n))$ the unique solution of the following BSDE
     (with finite horizon):
   \begin{align}\label{eq:bsde-n}
         Y(\tau;x,n)= \int_\tau^n( \psi(\bx(\sigma;x),Z(\sigma))-\lambda Y(\sigma))\dd \sigma+\int_\tau^n Z(\sigma;x,n)\dd W(\sigma), 
   \end{align}
then $|Y(\tau;x,n)| \leq \frac{M}{\lambda}$ and the following convergence rate holds:
   \begin{align*}
         |Y(\tau;x,n)-Y(\tau;x)| \leq \frac{M}{\lambda} \exp \left\{ -\lambda (n-\tau)\right\}.
   \end{align*}
    Moreover, 
   \begin{align*}
         \bE \int_0^\infty e^{-2\lambda \sigma} \| Z(\sigma;x,n)-Z(\sigma;x)\|^2 \dd \sigma \ \to \ 0.
   \end{align*}
   \item For all $T>0$ and $p>1$, the map $x \mapsto (\left.Y(\cdot;x)\right|_{[0,T]},\left.Z(\cdot;x)\right|_{[0,T]})$ is continuous from $X_\eta$ to the space $L^p(\Omega,C([0,T];\R)) \times L^p(\Omega,L^2([0,T];\Xi^\star))$.

\end{enumerate}
\end{prop}


We need to study the regularity of $Y(\cdot;x)$. More precisely we would like to show that $Y(0;x)$ is G\^ateaux differentiable with to respect the initial condition $x$ and that both $Y(0;x)$ and $\nabla_x Y(0;x)$ turn out to be bounded. 

The following result is one of the crucial point of the paper.
\begin{thm}\label{thm:GatY}
    Under Hypotheses \ref{hp:a,A,f,g,r,W}, \ref{hp:expdec},
     \ref{hp:l} and \ref{hp:grad-psi}, the map $x\mapsto Y(0;x)$ is G\^ateaux differentiable. Moreover, 
$|Y(0;x)|+|\nabla_x Y(0;x)|\leq c$. 
\end{thm}

\begin{proof}
   The argument follows essentially the proof of \cite[Theorem 3.1]{HuTe/2006}. 
   An important point is the boundedness of 
   $\nabla_x \bx (t,x) [h]$ ($\P$-a.s. and for any $t\geq 0$), which was proved in Proposition \ref{prop:dif-state-eq}. 
   We first recall that, under our assumptions (see \cite[Proposition 5.2]{FuTe/2002}), the map $x\mapsto (Y(t;x,n),Z(t;x,n))$ considered in Proposition \ref{prop:bsde} is G\^ateaux differentiable from $X_\eta$ to $L^p(\Omega;C(0,T;\R) \times L^p(\Omega; L^2(0,T;\Xi^\star))$ for all $p\geq 2$. Denoting by $\nabla_x Y(t;x,n)[h],\nabla_x Z(t;x,n)[h]$ the partial G\^ateaux derivative with respect to $x$ in the direction $h\in X$, the processes 
$(\nabla_x Y(t;x,n)[h],\nabla_x Z(t;x,n)[h])_{t\in [0,n]}$ solves the equation,
$\PP$-a.s. 
\begin{multline*}
        \nabla_x Y(\tau;x,n)[h]= \int_\tau^n( \psi(\bx(\sigma;x),Z(\sigma;x,n))\nabla_x \bx(\sigma;x)\dd \sigma\\
        -\int_\tau^n\lambda \nabla_x Y(\sigma;x,n)[h]\dd \sigma+\int_\tau^n Z(\sigma;x,n)\dd W(\sigma), 
\end{multline*}
   We notice that by the assumption made on $\nabla_x \psi, \nabla_z \psi$ and Proposition \ref{prop:diff-state-eq-gen}, we have
  \begin{align*}
      |\nabla_x \bx(t;x) h| \leq C |h| \qquad and \qquad |\nabla_z \psi(\bx(t;x),Z(t;x))| \leq C.
  \end{align*}
  Therefore by the same arguments based on Girsanov transform as in \cite[Lemma 3.1]{BrHu}, we obtain
  \begin{align*}
     \sup_{t\in [0,n]} |\nabla_x Y(t;x,n)| \leq C|h|, \quad \PP-a.s.
  \end{align*}
  and, again as in the proof of \cite[Lemma 3.1]{BrHu}, applying It\^o
  formula to $e^{-2\lambda t}|\nabla_x Z(t;x,n)|^2$, we get:
  \begin{align*}
      \mathbb{E} \int_0^\infty e^{-2\lambda t}(|\nabla_x Y(t;x,n)h|^2 + 
      |\nabla_x Z(t;x,n)h|^2 ) \dd t < C |h|^2.
  \end{align*}
  Let now $\mathcal{M}^{2,-2\lambda}$ be the Hilbert space of all couples of $(\cF_t)_{t\geq 0}$-adapted processes $(y,z)$, where $y$ has values
  in $\R$ and $z \in \Xi^\star$, such that
     \begin{align*}
      \mathbb{E} \int_0^\infty e^{-2\lambda t}(|y(t)|^2 + 
      |z(t)|^2 ) \dd t < C |h|^2.
  \end{align*}
  Fix now $x \in X_\eta$ and $h\in X$. Then there exists a subsequence 
  of $$\left\{ (\nabla_x Y(t;x;n)h,\nabla_x Z(t;x,n)h,\nabla_x Y(0;x;n)h)\right\}$$
  such that $\left\{ (\nabla_x Y(t;x;n)h,\nabla_x Z(t;x,n)h)\right\}$ converges weakly to 
  $(U^1(x,h),V^1(x,h))$ in $\mathcal{M}^{2,-2\lambda}$ and $\nabla_x Y(0;x,n)h$
  converges to $\xi(x,h) \in \R$. Proceeding as in \cite[Theorem 3.1]{HuTe/2006}, we see that the convergence of $\nabla_x Y(t;x;n)h$ is, in reality, in $L^2(0;T;\R)$ for all $T>0$ and, moreover, that $\lim_{n\to \infty} \nabla_x Y(0;x;n)h$ exists and coincides with the value at $0$ of the process 
  $>U(0;x)h$ defined by the following equation:
  \begin{align*}
     U(t;x)h = U(0;x)h& - \int_0^t \nabla_x \psi(\bx(s;x),Z(s;x))\nabla_x \bx(s;x) h\dd s \\&+\int_0^t \nabla_z \psi(\bx(s;x),Z(s;x)) V^1(s;x)[h] \dd t\\
      &   -\lambda\int_0^t  U(s;x)h\dd s+ V^1(s;x)[h] \dd W(s). 
  \end{align*} 
  Morever it hold that $U^1(t;x)h=U(t;x)h$ for any fixed $h\in X$.
  Summurizing up, we have that $U(0;x)h= \lim_{n\to \infty}\nabla_x  Y(0;x,n)h$ exists, it is linear and verifies $U(0;x)h\leq C|h|$ for every
  $h$ fixed. Finally, it is continuous in $x$ for every $h$ fixed.  
  Finally, for $t>0$ we have
  \begin{align*}
     &\lim_{t\to 0} \frac{Y(0;x+th)-Y(0;x)}{t} = \lim_{t\to 0} \lim_{n\to \infty}
     \frac{Y(0;x+th,n)-Y(0;x,n)}{t} \\=&  \lim_{t\to 0} \lim_{n\to \infty}
     \int_0^1 \nabla_x Y(0;x+th\theta)h \dd \theta \\= &
     \lim_{t\to o} \int_0^1 Y(0;x+th\theta) \dd \theta = U(0;x)h,
  \end{align*}
  and the claim is proved. 
\end{proof}

 Starting from the G\^ateaux derivatives of $Y$ and $Z$, we introduce suitable auxiliary processes which allow ourselves to express $Z$ in terms of $\nabla Y$ and $(I-B)^{1-\theta}$ and then get the fundamental relation for the optimal control problem introduced in Section \ref{sec:intro}. 
 The main point is to prove that the mappings $(h,x) \mapsto\nabla_x Y(t;x) (I-B)^{1-\theta}P g[h]$ and $(h,x) \mapsto \nabla_x Z(t;x) (I-B)^{1-\theta}Pg$
 are well defined as operators from $X \times X$ respectively
 in $L^\infty(\Omega;C([0,T];\R))$ and $L^\infty(\Omega;C([0,T];\Xi^\star))$.
       \begin{prop}\label{prop:est}
          For every $p\geq 2$, $\beta<0$, $x\in X_\eta$, $h\in X$ there exists two processes
        \begin{align*}
                \left\{ \Pi(t;x)[h]: \ t\geq 0\right\} \qquad \textrm{and} \quad \left\{ Q(t;x)[h]: \ t\geq 0\right\} 
         \end{align*}
           with $\Pi(t;x)[h]\in L^p(\Omega;C([0,T];X_\eta))$ and $Q(\cdot;x)[h] \in L^p(\Omega;C([0,T];\Xi^\star))$ for any $T>0$ and such that if $x\in X_\eta$ 
      then $\PP$-a.s. the following identifications hold:
     \begin{align}
                 &\Pi(t;x)[h]=\begin{cases}
                           \nabla_x Y(t;x)(I-B)^{1-\theta}[h], &  t> 0;\\
                               \nabla_x Y(0;x)(I-B)^{1-\theta}[h], & t=0;
                \end{cases}\label{eq:Pi} \\
               & Q(t;x)[h]=\begin{cases}
                           \nabla_x Z(t;x)(I-B)^{1-\theta}[h], & t> 0;\\
                              0, & t=0;
                \end{cases}\label{eq:Q}
    \end{align}
      Moreover, he map $(x,h) \mapsto \Pi(\cdot;x)[h]$ is continuous from $X_\eta \times X$ to 
        $L^p(\Omega; C_\beta([0,T];\R))$ and the map $(x,h) \mapsto Q(\cdot;x)[h]$ is continuous from $X_\eta \times X$ into $L^p(\Omega; C_\beta([0,T];\Xi^\star)))$ and both maps are linear with respect to $h$. Finally, there exists a positive constant $C$ such that
            \begin{equation}\label{eq:stima-est}
                    \bE  \|\Pi(0;x)\| 
                 \leq C 
            \end{equation}
       \end{prop}
       \begin{proof}
            For any $n\in \bN$, we introduce a suitable stochastic differential equation on $[0,n]$ which should give a sequence $(\Pi^{(n)}(\cdot;x)[h],Q^{(n)}(\cdot;x)[h])$ of approximating processes for 
            $(\Pi(\cdot;x),Q(\cdot;x))$; more precisely,
      for fixed $p\geq 2$, $x\in X_\eta$ and $h\in X$ we consider the equation
      \begin{equation}\label{eq:bsdePiQn}
  \begin{aligned}
       \Pi^{(n)}(t;x)[h] =\int_t^n \nu^{(n)}(t;x) h\dd t +\int_t^n \nabla_z \psi(\bx(t;x),Z(t;x)) Q^{(n)}(t;x)[h] \dd t\\
        \qquad \qquad \qquad   -\lambda\int_t^n  \Pi^{(n)}(t;x)[h]\dd t+ \ Q^{(n)}(t;x)[h] \dd W(t)  \qquad \qquad t\in [0,n]\\
\end{aligned}
\end{equation}
where 
\begin{align*}
     \nu^{(n)}(t;x) h&= {\mathbf 1}_{[0,n]}(t) \nabla_x \psi(\bx(t;x),Z(t;x))\nabla_x \bx(t;x) (I-B)^{1-\theta} P g h.
\end{align*}
  The solution to \eqref{eq:bsdePiQn} exists 
   (see \cite[Proposition 7.5]{CoMa11}) and the maps $x\mapsto \Pi^{(n)}(\cdot;x)[h]$ are continuous from 
    $ X_\eta \times X$ to $L^p(\Omega;C([0,n]; X_\eta)$ and $Q^{(n)}(\cdot;x)[h] \in L^p(\Omega; C([0,n];\Xi^\star))$. Further, if 
    $(Y(\cdot;n,x), Z(\cdot;n,x)$ are the processes introduced in Proposition \ref{prop:bsde}, then 
the following identifications hold
     \begin{align}
                 &\Pi(t;x)^{(n)}[h]=\begin{cases}
                           \nabla_x Y(t;x,n)(I-B)^{1-\theta}[h], &  t\in [0,n];\\
                               \nabla_x Y(0;x,n)(I-B)^{1-\theta}[h], & t=0;
                \end{cases} \label{eq:Pin}\\
               & Q(t;x)^{(n)}[h]=\begin{cases}
                           \nabla_x Z(t;x,n)(I-B)^{1-\theta}[h], & t\in [0,n];\\
                              0, & t=0;
                \end{cases}\label{eq:Qn}
    \end{align}
   for any  $x\in X_\eta$ and $h$ in $X$.
  Hence for any $n\in \bN$, $(\Pi^{(n)}(t;x),Q^{(n)}(t;x))_{t\in [0,n]}$ extend the mappings
  \begin{align*}
     h\mapsto  \nabla_x Y(t;x,n)(I-B)^{1-\theta}Pg h \qquad  h\mapsto  \nabla_x Z(t;x,n)(I-B)^{1-\theta}Pg h, \qquad t\in [0,n].
  \end{align*}
  Moreover, 
   we have the estimates
  \begin{multline}\label{eq:stima-dim}
                    \bE \sup_{t\in [0,n]} \|\Pi^{(n)}(t;x)\|^p_\eta e^{-p\beta t} + 
                 \bE \left(\int_0^n e^{-2\beta r}|\Pi^{(n)}(r;x)[h]|^2 \dd r\right)^{p/2} \\
                   + \bE \left(\int_0^n e^{-2\beta r}\|Q^{(n)}(r;x)[h]\|^2_{\Xi^\star} \dd r\right)
                 <\infty,
            \end{multline}
            for suitable $\beta >0$.
  It remains to prove that the processes $(\Pi^{(n)}(\cdot;x),Q^{(n)}(\cdot;x))$
  converge to some pair $(\Pi(\cdot;x),Q(\cdot;x))$ and that the identifications \eqref{eq:Pi} and \eqref{eq:Q} hold. For the convergence of
  $(\Pi^{(n)}(\cdot;x),Q^{(n)}(\cdot;x))$, we notice that $(\Pi^{(n)}(\cdot;x),Q^{(n)}(\cdot;x))$ solves a BSDE whit bounded  coefficients. In fact, by the assumption made on $\nabla_x \psi, \nabla_z \psi$ and Proposition \ref{prop:diff-state-eq-gen}, we have
  \begin{align*}
      |\nu^{(n)}(t;x) h| \leq C |h| \qquad and \qquad |\nabla_z \psi(\bx(t;x),Z(t;x))| \leq C.
  \end{align*} 
  Hence, following \cite[Theorem 3.1]{HuTe/2006} or Theorem \ref{thm:GatY} above, we conclude that $\Pi(0;x)h= \lim_{n\to \infty} \Pi^{(n)}(0;x)h$ exists, it is linear, verifies $|\Pi(0;x)h|\leq C|h|$ for every
  $h$ fixed and it is continuous in $x$ for every $h$ fixed.  
  Finally, since on $X_{\theta-1}$ the processes $(\Pi^{(n)}(\cdot;x),Q^{(n)}(\cdot;x))$ extends the processes
  $$(\nabla_x Y(t;x,n)(I-B)^{1-\theta}[h],\nabla_x Z(t;x,n)(I-B)^{1-\theta}[h]),$$ on the same space we have
  \begin{align*}
       \Pi(0;x)h = \lim_{n\to \infty} \nabla_x Y(0;x,n)(I-B)^{1-\theta}Pg[h]
       = \nabla_x Y(0;x)(I-B)^{1-\theta}[h]
  \end{align*}

      \end{proof}
       
        \begin{cor}\label{cor:Y-v}
        Setting $v(x)=Y(0;x)$, we have that $v$ is G\^ateaux differentiable with respect to 
      $x$ on $X_\eta$ and the map $(x,h) \mapsto \nabla v(x)[h]$ is continuous. 
     
      Moreover, for $x\in X_\eta$ the linear operator $h\mapsto \nabla v(x)(I-B)^{1-\theta}[h]$
     - a priori defined for $h\in \cD$ - has an extension to a bounded linear operator from $X$ into $\R$, that we denote by $[\nabla v (I-B)^{1-\theta}](x)$.

      Finally, the map $(x,h) \mapsto [\nabla v (I-B)^{1-\theta}Pg](x)$ is continuous as a mapping from $X_\eta \times X$ into $\R$ and there exists $C>0$ such that 
      \begin{align}\label{eq:pol-growth}
          |[\nabla v(I-B)^{1-\theta}](x)[h]|_X \leq C|h|_X,
     \end{align} 
      for $x\in X_\eta, h\in X$.
    \end{cor}
     
     \begin{proof}
          First of all, we notice that $Y(0;x)$ is deterministic, since $(\cF_t)_{t\geq 0}$ is generated 
       by the Wiener process $W$ and $Y(\tau)$ is $\cF_\tau$-adapted.
         
         Moreover, in Theorem \ref{thm:GatY} we proved that the map $x\mapsto Y(\cdot;x)$ is continuous and G\^ateaux differentiable with values in $L^p(\Omega;C([0,T];\R))$; consequently, it follows that $x\mapsto v(x)=Y(0;x)$ is G\^ateaux differentiable with values in 
     $\R$.
         
      Next, we notice that $\Pi(0;x)=\nabla_x Y(0;x)(I-B)^{1-\theta}[h]$. The existence of the required extension and its continuity are direct consequence of Proposition \ref{prop:est} and estimate \eqref{eq:pol-growth} follows from \eqref{eq:stima-est}.
        
     \end{proof}
     We are now in the position to give a meaning to the expression $\nabla_x Y(t;x)(I-B)^{1-\theta}$ and, successively, to identify it with the process $Z(t;x)$. To this end we quote from \cite{CoMa11} a preliminary result where
     we investigate the existence of the joint quadratic variation of $W$ with a process of the form $\left\{w(t,\bx(t)): \ t\in [0,T]\right\}$ for a given function $w: [0,T]\times X \to \R$, on an interval $[0,s] \subset [0,T)$. In order to simplify the exposition we omit the proof. We only notice
     that the proof requires the study of the Malliavin derivative of $\bx$ and this can be done in the same way as in \cite[Section 6.1]{CoMa11}.
     \begin{prop}\label{prop:quad-var}
Suppose that $w\in C([0,T)\times X_\eta;\R)$ is G\^ateaux differentiable with respect to $\bx$,
and that for every $s<T$ there exist constant $K$ (possibly depending on $s$)
such that 
\begin{align}\label{eq:stima-w}
     |w(t,x)|\leq K, \quad |\nabla w(t,x)|\leq K, \quad t\in [0,s], x\in X.
\end{align}
Let $\eta$ and $\theta$ satisfy condition \eqref{eq:eta-theta} in Theorem \ref{t:state space setting}. 
Assume that for every $t\in [0,T)$, $x\in X_\eta$ the linear operator $k\mapsto \nabla w(t,x)(I-B)^{1-\theta}k$ (a priori defined for $k\in \mathcal{D}$) has an extension to a bounded linear operator $X \to \R$, that we denote by $[\nabla w(I-B)^{1-\theta}](t,x)$. Moreover, assume that the map
$(t,x,k) \mapsto [\nabla w(I-B)^{1-\theta}](t,x)k$ is continuous from $[0,T) \times X_\eta \times X $ into $\R$. 
For $t\in [0,T),x\in X_\eta$, let $\left\{ \bx(t;s,x),t\in[s,T]\right\}$ be the solution of equation 
\eqref{eq:un-state-eq}. Then the process $\left\{w(t,\bx(t;s,x)), t\in [s,T]\right\}$ admits a joint quadratic variation process with $W^j$, for every $j\in \bN$, on every interval $[s,t] \subset [s,T)$, given by
\begin{align*}
     \int_s^t [\nabla w(I-B)^{1-\theta}](r,\bx(r;s,x))(I-B)^{\theta} P g e_j\dd r.
\end{align*}
\end{prop}

    the above result allows to identify the process $Z(\cdot;x)$ with $[\nabla v (I-B)^{1-\theta}](\bx(\cdot;x))(I-B)^\theta P g$, as we can see in the result below.
      \begin{cor}\label{cor:quad-var}
     For every $x\in X_\eta$ we have, $\PP$-a.s.
     \begin{align}
          & Y(t;x)=v(\bx(t;x)), \qquad for \ all \ t\geq 0;\label{eq:Y-v}\\
           & Z(\cdot,x)=[\nabla v (I-B)^{1-\theta}](\bx(\cdot;x))(I-B)^\theta P g(\bx(\cdot;x)), \ 
         for \ almost \ all \ t \geq 0. \label{eq:Z-v}
    \end{align}
     \end{cor}
    \begin{proof}
         We need to consider the equation (which is slightly more general than \eqref{eq:state-eq})
        \begin{align}\label{eq:forw-est}
             \begin{cases}
    \dd \bx(t)= B \bx(t)\dd t+ (I-B)P \, f( J \bx(t))\dd t \\
       \qquad \qquad \qquad + (I-B)P g(u(t)) (r(J \bx(t),\gamma(t))\dd t+\dd W(t))\\
     \bx(s)=x,
   \end{cases}
        \end{align}
        for $t$ varying $ [s,\infty)$. We set
            $\bx(t)=x$ for $t\in [0,s)$ and we denote by $(\bx(t;s,x))_{t\geq 0}$ the solution, to indicate 
            dependence on $x$ and $s$. By an obvious extension of the results concerning the BSDE \eqref{eq:bsde2} stated above, we can solve the backward equation \eqref{eq:bsde2} with 
             $\bx$ given by \eqref{eq:forw-est}; we denote the corresponding solution
              by $(Y(t;s,x),Z(t;s,x))$ for $t\geq 0$. Thus $(Y(t;x),Z(t;x))$ coincides with 
             $(Y(t;0,x),Z(t;0,x))$ that occurs in the statement of Proposition \ref{prop:bsde}. 

             Now let us prove \eqref{eq:Y-v}. 
            We start from the well-known equality: for $0\leq r \leq T$, $\PP$-a.s.
            \begin{align*}
                  \bx(t;s,x)=\bx(t;r,\bx(r;s,x)), \qquad \textrm{for all } t\in [s,T].
           \end{align*}
          It follows easily from the uniqueness of the backward equation \eqref{eq:bsde2}
           that $\PP$-a.s.
          \begin{align*}
                 Y(t;s,x)=Y(t;r,\bx(r;s,x)), \qquad \textrm{for all }\ t\in [s,T].
          \end{align*}
             In particular, for every $0\leq r \leq s \leq t $,
            \begin{align*}
                  &Y(t;s,\bx(s;r,x))=Y(t;r,x), \qquad \textrm{for } t\in [s,\infty),\\
                  &Z(t;s,\bx(s;r,x))=Z(t;r,x), \qquad \textrm{for } t\in [s,\infty).
           \end{align*}
         Since the coefficients of \eqref{eq:forw-est} do not depend on time, we have
         \begin{align*}
                \bx(\cdot;0,x)\stackrel{(d)}{=}\bx(\cdot + t;t,x), \qquad t\geq 0,
        \end{align*}
       where $\stackrel{(d)}{=} $ denotes equality in distribution. As a consequence
          we obtain $$
        (Y(\cdot;0,x),Z(\cdot;0,x))\stackrel{(d)}{=} (Y(\cdot+t;t,x),Z(\cdot+t;t,x)), \qquad t\geq 0,
           $$
         where both sides  of the equality are viewed as random elements with values in the space
          $C(\R_+;\R)\times L^2_{loc}(\R_+;\Xi^\star)$. In particular
         $$
                   Y(0;0,x)\stackrel{(d)}{=}Y(t;t,x),
         $$
        and since they are both deterministic, we have
         $$
                   Y(0;0,x)=Y(t;t,x), \qquad x\in X_\eta, t\geq 0
         $$
          so that we arrive at \eqref{eq:Y-v}.
  
            To prove \eqref{eq:Z-v} we consider the joint quadratic variation of $(Y(t;x))_{t\in [0,T]}$
         and $W$ on an arbitrary interval $[0,t] \subset [0,T)$; from the backward equation 
          \eqref{eq:bsde2} we deduce that it is equal to $\int_0^t Z(r;x)\dd r$. On the other side, 
          th same result can be obtained by considering the joint quadratic variation of $(v(\bx(t;x)))_{t\in [0,T]}$ and $W$. Now by an application of Proposition \ref{prop:quad-var}
      (whose assumptions hold true by Corollary \ref{cor:Y-v}) leads to the identity
          \begin{align*}
                    \int_0^t Z(r;x)\dd r = \int_0^t [\nabla v (I-B)^{1-\theta}](r;\bx(r;x))(I-B)^\theta P g(\bx(r;x))\dd r,
          \end{align*}
         and \eqref{eq:Z-v} is proved.
    \end{proof}

\section{The stationary Hamilton Jacobi Bellman equation}\label{sec:HJB}

Now we proceed as in \cite{FuTe/2004}. Let us consider again the solution $\bx(t;x)$ of equation \eqref{eq:state-eq-2} and denote by 
$(P_{t})_{t\geq 0}$ its transition semigroup:
$$
      P_{t}[h](x)=\bE h(\bx(t;x)), \qquad x\in X_\eta, 0 \leq t,
$$
for any bounded measurable $h: X \to \R$. We notice that by the bound \eqref{eq:exuni} this formula is true for every $h$ with polynomial growth. In the following $P_{t}$ will be considered as an operator acting on this class of functions.

Let us denote by $\cL$ the generator of $P_{t}$:
\begin{align*}
     \cL [h](x)=\frac{1}{2}{\rm Tr} [(I-B)Pg\nabla^2h(x)g^*p^*(I-B)^*]
    + \langle Bx+(I-B)P f(Jx),\nabla h(x)\rangle,
\end{align*}
where $\nabla h$ and $\nabla^2 h$ are first and second G\^ateaux derivatives 
of $h$ at the point $x\in X$ (here we are identified with elements of $X$ and $L(X)$ respectively).
This definition is formal, since it involves the terms $(I-B)Pg$ and $(I-B)P f$ which - {\emph a priori} - are not defined as elements of $L(X)$ and the domain of $\cL$ is not specified.

In this section we address solvability of the nonlinear stationary Kolmogorov equation:
\begin{align}\label{eq:HJB1}\tag{$HJB$}
        \cL[v(\cdot)](x)=\lambda v(x)- \psi(x,\nabla v(x)(I-B)Pg), \quad  \ x\in X.
\end{align}
This is a nonlinear elliptic equation for the unknown function $v: X_\eta \to \R$. 
We define the notion of solution of the \eqref{eq:HJB} by means of the variation of constant formula:
\begin{defn}\label{def:solHJB}
    We say that a function $v: X_\eta \to \R$ is a mild solution of the Hamilton - Jacobi - Bellman equation \eqref{eq:HJB} if the following conditions hold:
\begin{enumerate}
    \item $v\in C( X;\R)$ and there exist $C\geq 0$ such that 
    $|v(x)|\leq C, \ x\in X$;
    \item $v$ is G\^ateaux differentiable with respect to $x$ on $ X_\eta$
     and the map $(x,h)\mapsto \nabla v(x)[h]$ is continuous $X_\eta \times X \to \R$;
    \item \label{it:boundv} For all $x\in X_\eta$ the linear operator $k \mapsto \nabla v(x)(I-B)^{1-\theta} k$ (a priori defined for $k\in \mathcal{D})$ has an extension to a bounded linear operator on $X \to \R$, that we denote by $[\nabla v(I-B)^{1-\theta}](x)$.
   
Moreover the map $(x,k)\mapsto [\nabla v (I-B)^{1-\theta}](x)k$ is continuous $ X_\eta \times X \to \R$ and there exist constants $C\geq 0$ such that
\begin{align*}
   \| [\nabla v (I-B)^{1-\theta}](x)\|_{L(X)} \leq C, \qquad x\in X_\eta.
\end{align*}
  \item The following equality holds for every $x\in X_\eta$:
    \begin{align}\label{eq:solvar}
         v(x)=e^{-\lambda T}P_{T}[v](x)-\int_0^T e^{-\lambda s}P_{s} \left[ \psi(  [\nabla v (I-B)^{1-\theta}](\cdot)(I-B)^\theta P g\right])(x) \dd s.
    \end{align}
\end{enumerate} 
\end{defn}

\begin{rem}
    We notice that, by assumption, $|\psi(x,z)|\leq C$; moreover, we saw in the preceeding sections how to give a meaning to the term
    \begin{align*}
       \nabla v(x) (I-B) P g.
    \end{align*} 
 Hence if $v$ is a function satisfying the bound required in Definition \ref{def:solHJB},\ref{it:boundv}
  we have
\begin{align*}
     \left| \psi(x,[\nabla v (I-B)^{1-\theta}](x)(I-B)^{\theta} P g )\right|\leq C
\end{align*}
and formula \eqref{eq:solvar} is meaningful.
\end{rem}

Now we are ready to prove that the solution of the equation \eqref{eq:HJB} can be defined by means of the solution of the BSDE associated with the control problem \eqref{eq:state-eq-2}.
\begin{thm}\label{thm:ident}
    Assume Hypotheses \ref{hp:a,A,f,g,r,W}, \ref{hp:dato-in}, \ref{hp:l} and \ref{hp:grad-psi}; then there exists a unique mild solution of the \ref{eq:HJB} equation. The solution is given by the formula
\begin{align}\label{eq:defv}
     v(x)=Y(0;0,x)=Y(t;t,x), \qquad t\geq 0,
\end{align} 
where $(\bx,Y,Z)$ is the solution of the forward-backward system \eqref{eq:state-eq-2} and \eqref{eq:bsde2}. 
\end{thm}

\begin{proof}
     We start by proving existence. In particular, we prove that $v$, given by 
\eqref{eq:defv}, is a solution of \ref{eq:HJB}. Hence, we set: 
\begin{align*}
   v(x):=Y(0;0,x).
\end{align*}
By Corollary \ref{cor:quad-var} the function $v$ has the regularity properties stated in Definition \ref{def:solHJB}. In order to verify that equality \eqref{eq:solvar} holds we first fix $x\in X_\eta$. We notice that 
\begin{equation*}
     \psi(\cdot, [\nabla v (I-B)^{1-\theta}](\cdot)(I-B)^{\theta}P g)(x) 
=
         \psi( \cdot, [\nabla Y (I-B)^{1-\theta}](\cdot)(I-B)^{\theta}P g)(x)
\end{equation*}
and we recall that 
\begin{align*}
     [\nabla v(I-B)^{1-\theta}](\bx(t;0,x))(I-B)^{\theta}Pg=Z(t;0,x), \qquad 0\leq  t.
\end{align*}
Hence
\begin{multline}\label{eq:psiPtT}
     P_{t}\left[ \psi( \cdot, [\nabla v (I-B)^{1-\theta}](\cdot)(I-B)^{\theta}P g)\right](\bx(t;0,x)) \\
  = \bE \left[ \psi(\bx(t;0,x),  Z(t;0,x))\right].
\end{multline}
On the other hand, applying the It\^o formula to the backward equation gives
\begin{multline*}
    e^{-\lambda t}Y(t;t,x)-e^{-\lambda T}Y(T;t,x)+\int_t^T e^{-\lambda r}Z(r;t,x)\dd W(r)\\
      = -\int_t^T e^{-\lambda r} \psi(\bx(r;t,x),Z(r;t,x)) \dd r,
\end{multline*}
for any $0\leq t \leq T <\infty$.
Taking the expectation and recalling again Corollary \ref{cor:quad-var} we obtain 
\begin{align*}
   e^{-\lambda t} v(x)=e^{-\lambda T}\bE v(\bx(T;t,x))-\bE \int_t^T e^{-\lambda r}\psi(\bx(r;t,x),Z(r;t,x)) \dd r 
\end{align*}
and substituting in the integral the expression obtained in \eqref{eq:psiPtT} we get the required equality \eqref{eq:solvar}.
This completes the proof of the existence part.

Now we consider uniqueness of the solution. Let $v$ denote a mild solution. We look for a convenient expression for the process $v(\bx(r;t,x))$, $0\leq t \leq r \leq T <\infty$. By \eqref{eq:solvar}
and the definition of $P_{t}$, for any $y\in X$ we have
\begin{equation}\label{eq:12set}
\begin{aligned}
   & v(y)
    = e^{-\lambda (T-t)}\bE \left[v(\bx(T-t;0,y))\right]\\
  &-\int_0^{T-t} 
      e^{-\lambda r}\bE \left[ \psi(\bx(r;0,y),[\nabla v (I-B)^{1-\theta}] (r,\bx(r;0,y))(I-B)^{\theta}P g )\right]
    \dd r. 
\end{aligned}
\end{equation}
Set $y=\bx(t;0,x)$. Then equality \eqref{eq:12set} rewrites as
\begin{multline*}
    v(\bx(t;0,x))
    = e^{-\lambda (T-t)}\bE \left[v(\bx(T-t;0,\bx(t;0,x)))\right]\\
  -\int_0^{T-t} 
      e^{-\lambda r}\bE \left[ \psi(\bx(r;0,\bx(t;0,x)),[\nabla v (I-B)^{1-\theta}] (\bx(r;0,\bx(t;0,x)))(I-B)^{\theta}P g )\right]
    \dd r. 
\end{multline*}
Moreover, 
recalling that for any $r\in [t,T]$ the equality 
$$
   \bx(r;t,\bx(t;0,x))=\bx(r;0,x)
$$ hold $\PP$-a.s.  we obtain
\begin{multline*}
    v(\bx(t;0,x))
    = e^{-\lambda (T-t)}\bE \left[v(\bx(T;t,x))\right]\\
  -\int_0^{T-t} 
      e^{-\lambda r}\bE \left[ \psi(\bx(r+t;0,x),[\nabla v (I-B)^{1-\theta}] (\bx(r+t;0,x))(I-B)^{\theta}P g )\right]
    \dd r. 
\end{multline*}
 Since $\bx(t;0,x)$ is $\cF_t$-measurable, we can replace the expectation by the conditional expectation given $\cF_t$:
\begin{multline*}
    e^{-\lambda t}v(\bx(t;0,x))
    = e^{-\lambda T}\bE ^{\cF_t}\left[v(\bx(T;t,x))\right]\\
  -\bE^{\cF_t}\left[ \int_t^T 
       e^{-\lambda r}\psi(\bx(r+t;0,x),[\nabla v (I-B)^{1-\theta}] (\bx(r+t;0,x))(I-B)^{\theta}P g 
    \dd r\right]. 
\end{multline*}
and, by change of variable,
we get:
\begin{align*}
    e^{-\lambda t} v(\bx(t;0,x)) &=e^{-\lambda T} \bE ^{\cF_t}\left[v(\bx(T;t,x))\right] \\
     &  -  \bE ^{\cF_t}\left[\int_t^T 
     \psi(\bx(r;0,x),[\nabla v (I-B)^{1-\theta}] (\bx(r;0,x))(I-B)^{\theta}P g 
    \dd r\right]\\
     &= \bE^{\cF_t}[\xi] \\
     &+\bE^{\cF_t} \left[\int_0^t  \psi( \bx(r;0,x),[\nabla v (I-B)^{1-\theta}] (\bx(r;0,x))(I-B)^{\theta}P g 
    \dd r\right],
\end{align*}
where we have defined 
\begin{multline*}
     \xi:= e^{-\lambda T}v(\bx(T;t,x))\\
     - \int_0^T  \psi( \bx(r;0,x),[\nabla v (I-B)^{1-\theta}] (\bx(r;0,x))(I-B)^{\theta}P g )\dd r.
\end{multline*}
Since we assume polynomial growth for $v$ and $\nabla v$, therefore $\xi$ is square integrable. Since
$(\cF_t)_{t\geq 0}$ is generated by the Wiener process $W$, it follows that there exists $\tilde{Z} \in L^2_\cF(\Omega\times [0,T];L_2(\Xi;\R))$ such that 
  $$\bE^{\cF_t}[\xi]=\bE[\xi]+\int_0^t \tilde{Z}(r)\dd W(r), \qquad t \in [0,T].
$$
  An application of the It\^o formula gives
\begin{multline}\label{eq:decomp}
    v(\bx(t;0,x))=\bE[\xi]+\int_0^t e^{\lambda r}\tilde{Z}(r)\dd W(r) + \lambda \int_0^t v(\bx(r;0,x)\dd r\\
      + \int_0^t  \psi(r, \bx(r;t,x),[\nabla v (I-B)^{1-\theta}] (r,\bx(r;t,x))(I-B)^{\theta}P g 
    \dd r,
\end{multline}
 We conclude that the process $v(\bx(t;s,x))$, $t\in [s,T]$ is a real continuous semimartingale.

For $\xi \in \Xi$, let us define $W^\xi$ by $W^\xi(t)=\langle \xi,W(t)\rangle$ and let us consider the joint quadratic variation process of $W^\xi$ with both sides of \eqref{eq:decomp}. 
Applying Proposition \ref{prop:quad-var}, we obtain, $\PP$-a.s., 
\begin{align*}
    \int_0^t [\nabla v(I-B)^{1-\theta}](\bx(r;0,x))(I-B)^\theta P g \dd r= \int_0^t \tilde{Z}(r)\dd r.
\end{align*}
Therefore, for a.a. $t\in [0,T]$, we have $\PP$-a.s. $[\nabla v(I-B)^{1-\theta}](\bx(r;0,x))(I-B)^\theta Pg =\tilde{Z}(r)$, so substituting into \eqref{eq:decomp} we obtain, for $t\in [0,T]$,
\begin{multline*}
     v(\bx(t;0,x))= v(x) +\int_0^t [\nabla v(I-B)^{1-\theta}](\bx(r;t,x))(I-B)^\theta Pg \dd W(r)\\
    +\int_0^t \lambda v(\bx(r;0,x)
    \dd r+ \int_0^t \psi(\bx(r;t,x),[\nabla v (I-B)^{1-\theta}] (\bx(t;0,x))(I-B)^{\theta}P g 
    \dd r.
\end{multline*}
Comparing with the backward equation \eqref{eq:bsde2} we notice that the pairs 
\begin{align*}
    & (Y(t;0,x),Z(t;0,x))
\intertext{and}
  & (v(\bx(t;0,x)), [\nabla v(I-B)^{1-\theta}](\bx(r;0,x))(I-B)^\theta Pg =\tilde{Z}(r))
\end{align*}
solve the same equation. By uniqueness, we have $Y(t;0,x)=v(\bx(t;0,x)), t\in [0,T]$, and setting $t=0$ we obtain $Y(t;t,x)=v(x)$. 
\end{proof}

\section{Synthesis of the optimal control}\label{sec:syn-cont}
In this section we proceed with the study of the optimal control problem associated with the stochastic Volterra equation 
   \begin{align}\label{eq:Volterra-cont}
    \begin{cases}
       \frac{\dd}{\dd t} \int_{-\infty}^t a(t-s)u(s)\dd s= A u(t)+ f(u(t))\\
         \qquad \qquad \qquad \qquad + g(u(t)) \,[\, r(u(t),\gamma(t))
            + \dot{W}(t)\, ], \qquad t\in [0,T]\\
       u(t)=u_0(t), \qquad t\leq 0.
    \end{cases}
\end{align}
   for a process $u$ with values in the Hilbert space $H$. 
   Here $f$ and $r$ are the nonlinear functions introduced in Hypothesis \ref{hp:a,A,f,g,r,W} and $\gamma=\gamma(\omega,t)$ is the control variable, which is assumed to be 
a predictable real-valued process $\cF_t$-adapted. The optimal control that we wish treat 
  consists in minimizing over all admissible controls a cost functional of the form 
\begin{align}\label{eq:cost1}
    \bJ(u_0,\gamma)=\bE \int_0^\infty e^{-\lambda t}\ell(u(t),\gamma(t))\dd t,
\end{align}
where $\ell: \ H\times U \to \R$ is a given real-valued function.

We will work under assumptions \ref{hp:a,A,f,g,r,W} and \ref{hp:grad-psi}. 

%

To handle the control problem, we first restate equation \eqref{eq:Volterra-cont}
in an evolution setting and we provide the synthesis of the optimal control by using 
the forward backward system approach.

As it has been proved in Section \ref{sec:anal-set}, given a control process $\gamma$ and any $u_0\in H$,
we can rewrite the problem \eqref{eq:Volterra-cont} in the following abstract form
\begin{align*}
     \begin{cases}
         \dd \bx(t)= B \bx(t) \dd t + (I-B)f(J \bx(t)) \dd t  \\ \qquad \qquad \qquad  +
          (I-B)Pg(r(J\bx(t),\gamma(t)) \dd t + \dd W(t))\\
         \bx(0)=x. 
     \end{cases}
\end{align*}
Here $X$ is a suitable separable Hilbert space, $B$ is a densely defined sectorial operator
on a domain $D(B) \subset X$, $P$ is a linear operator from $H$ with values into a real interpolation space $X_\theta$ ($\theta \in (0,1)$) between $D(B)$ and $X$ and $J$ is a linear operator from $X_\eta$ ($\eta \in (0,1)$) into $H$; finally $x\in X_\eta$ (see Theorem \ref{t:state space setting} for more details).  

In this setting the cost functional will depend on $x$ and $\gamma$ and is given by
\begin{align}\label{eq:costacs}
      \bJ(x,\gamma)= \bE \int_0^\infty e^{-\lambda s} \ell(J\bx(s),\gamma(s))\dd s;
\end{align}
(with an abuse of notation we still denote the rewritten cost functional as $\bJ$).
We notice that for all $\lambda >0$ the cost functional is well defined and $\bJ(x,\gamma) <\infty$ for all $x\in X_\eta$ and all a.c.s.
There are different ways to give precise meaning to the above problem; one 
of them is the so called \emph{weak formulation} and will be specified below.

In the weak formulation the class of \emph{admissible control system} (a.c.s.) is given by
the set $\bU:= (\hat{\Omega},\hat{\cF},(\hat{\cF}_t)_{t\geq 0},\hat{\PP},\hat{W},\hat{\gamma})$, where
$ (\hat{\Omega},\hat{\cF},\hat{\PP})$ is a complete probability space; the filtration $(\hat{\cF}_t)_{t\geq 0}$ verifies the usual conditions, the process $\hat{W}$ is a Wiener process
with respect to the filtration $(\hat{\cF}_t)_{t\geq 0}$ and the control $\hat{\gamma}$ is an 
$\cF_t$-predictable process taking values in some subset $\cU$ of $X$ with respect to the filtration $(\hat{\cF}_t)_{t\geq 0}$. 
With an abuse of notation, for given $x\in X_\eta$, we associate to every a.c.s.
a cost functional $\bJ(x,\bU)$ given by the right side of \eqref{eq:costacs}. Although formally the
same, it is important to note that now the cost is a functional of the a.c.s. and
not a functional of $\hat{\gamma}$
 alone. Any a.c.s. which minimizes $\bJ(x,\cdot)$, if it exists, is
called optimal for the control problem starting from $x$ at time $t$ in the weak
formulation. The minimal value of the cost is then called the optimal cost.
Finally we introduce the value function $V: \ X_\eta \to  \bR$ of the problem as:
\begin{align*}
   V(x)=\inf_{\gamma\in \cU} \bJ(x,\gamma), \quad x\in X_\eta,
\end{align*}
where the infimum is taken over all a.c.s. $\bU$.

Optimal control problems on infinite horizon of the type considered in this paper (with coefficients having the properties listed in Hypotheses \ref{hp:a,A,f,g,r,W}, \ref{hp:l}, \ref{hp:grad-psi}) have
been exhaustively studied by Fuhrman and Tessitore in \cite{FuTe/2002,FuTe/2004} and Hu and Tessitore in \cite{HuTe/2006}, compare Theorem 5.1.
Within their approach the existence of an optimal control is related to the existence
of the solution of a suitable forward backward system (FBSDE) that is a system in which the coefficients of the  backward 
equation depend on the solution of the forward equation.
Moreover, the optimal control can be selected using a feedback law given in terms of the
solution to the corresponding FBSDE.

We recall that the Hamiltonian 
corresponding to our control problem is given by
\begin{align*}
     \psi(x,z)= \inf_{\gamma\in \cU}\left\{\ell(Jx,\gamma)+ z\cdot r(Jx,\gamma) \right\}, \qquad 
  x\in X_\eta \ z \in \Xi^\star,
\end{align*}
and we define the following set
\begin{equation}\label{eq:Gamma}
\begin{aligned}
   \Gamma(x,z)= \left\{\gamma\in \cU: \ell(Jx\gamma)+z\cdot r(Jx,\gamma)= \psi(x,z)\right\},\\
   \qquad \qquad \quad \qquad t\in [0,T], \ x \in X_\eta, \ z \in \Xi^\star.
\end{aligned}
\end{equation}

For further use we require an additional property of the function $\psi$:
\begin{hypothesis}\label{hp:psi}
    For all $x\in X_\eta$, $z\in \Xi^\star$ there exists a unique $\Gamma(x,z)$
that realizes the minimum in 
\eqref{eq:Gamma}. Namely:
     \begin{align*}
       \psi(x,z)=\ell(Jx,\Gamma(x,z))+z\cdot r(Jx,\Gamma(x,z)),
     \end{align*} 
    with $\Gamma \in C(X_\eta\times \Xi^\star;U)$.
\end{hypothesis}


Now,
let us consider an arbitrary set-up $(\tilde{\Omega},\tilde{\cF},\tilde{\PP},\tilde{W})$ and 
\begin{equation}\label{eq:cont2}
\begin{aligned}
   &\tX(t)=e^{tB} x +
\int_0^t e^{(t-\sigma)B}(I-B)Pf(\tX(\sigma)) \dd \sigma +\\
&\qquad \qquad\int_0^t e^{(t-\sigma)B}(I-B)P g\, \dd \tilde{W}(\sigma), \quad  0\leq t \leq T <\infty.
\end{aligned}
\end{equation}
By Theorem \ref{thm:exuni} stated in Section \ref{sec:exun-abs}, equation \eqref{eq:cont2} is well-posed and the solution $(\tX(t))_{t\geq 0}$ is a continuous process in $X_\eta$, adapted to the filtration $(\tF_t)_{t \geq 0}$.
Moreover, the law of $(\tilde{W}, \tX)$ is uniquely determined by $x$, $B$, $f$ and $g$. 
We define the process
\begin{align*}
   \tilde{W}^\bU(t) = \tilde{W}(t) - \int_0^t r(\tX(s), \tilde{\gamma}(s)) \dd s, \quad 0\leq t \leq T, 
\end{align*}
and we note that, since $r$ is bounded, by the Girsanov theorem there exists a probability measure $\PP $ on $(\Omega,\cF)$
such that $\tilde{W}^\bU$ is a Wiener process under $\PP$. Rewriting equation \eqref{eq:cont2}
in terms of $\tilde{W}^\bU$ we get that
$\tX$ solves the controlled state equation (in weak sense)
\begin{equation}
\begin{aligned}
&\tilde{\bx}(t) = x + \int_0^t e^{(t-\sigma)B}(I-B)P(\tilde{\bx}(\sigma)) \dd \sigma +\\
&\qquad\int_0^t e^{(t-\sigma)B}(I-B)Pg\dd \tilde{W}^\bU(\sigma) +
\int_0^t e^{(t-\sigma)\bA}(I-B)P g r(\tX(s) , \tilde{\gamma}(s)) \dd s.
\end{aligned}
\end{equation}

We notice that for all $\lambda >0$ the cost functional is well defined and $\bJ(x,\gamma)<\infty$ for all $x\in X_\eta$ and all a.c.s. $\bU$.

By Theorem \ref{thm:ident}, for all $\lambda >0$ the stationary Hamilton-Jacobi-Bellman equation relative to the above stated problem, namely:
\begin{align}\label{eq:HJB2}\tag{$HJB$}
        \cL v(x)=\lambda v(x)- \psi(x,\nabla v(x)(I-B)Pg), \quad  x\in X_\eta
\end{align}
admits a unique mild solution, in the sense of Definition \ref{def:solHJB}.
Here $\cL$ is the infinitesimal generator of the Markov semigroup corresponding to the process $\bx$:
\begin{align*}
    \cL \phi(x)= \frac{1}{2}{\rm Tr}((I-B)Pgg^*P^*(I-B)^*\nabla^2\phi(x)) + 
     \langle B x + (I-B)Pf(x),\nabla\phi(x)\rangle.
\end{align*}
Let $v$ be the unique mild solution of equation \eqref{eq:HJB2}. Consider the following finite horizon backward equation (with respect to probability $\tilde{P}$ and to the filtration generated by $\left\{ \tilde{W}_t:  \ t\in [0,T] \right\}$:

\begin{align}\label{eq:bsde}
  \tY(t) -v(\tX(T))+ \int_t^T \tZ \dd \tW(\sigma) = \int_t^T
\psi( \tX(\sigma), \tZ(\sigma))\dd \sigma- \lambda \int_t^T \tY(\sigma) \dd \sigma,
\end{align}
where $\psi$ is the Hamiltonian function.
It was proved in Section \ref{sec:backward} that there exists a solution $(\tX,\tY, \tZ)$ of the forward-backward system
\eqref{eq:cont2}-\eqref{eq:bsde} on the interval $[0, T]$, where $\tY$ is
unique up to indistinguishability and $\tZ$ is unique up to modification. Moreover from the proof
of Theorem 4.8\cite{FuTe/2002} it follows that the law of $(\tY,\tZ)$ is uniquely determined by the law of $(\tW, \tX)$
and $\psi$. To stress dependence on the initial datum $x$, we will denote the solution of \eqref{eq:cont2} and \eqref{eq:bsde}
by $\{ (\tilde{\bx}^{x}(t), \tY^{x}(t), \tZ^{x}(t)), t \in [0, T]\}$.
 Theorem \ref{thm:ident} and uniqueness of the solution of system \eqref{eq:cont2}-\eqref{eq:bsde}, yields that
\begin{align*}
     \tilde{Y}^x(t)=v(\tilde{\bx}^x(t)), \qquad \tZ^x(t)=[\nabla v(I-B)^{1-\theta}](\tilde{\bx}^x(t))(I-B)^\theta P g.
\end{align*}
Hence the solution of the \eqref{eq:HJB} equation is the value function of the control problem; this allows to construct the optimal feedback law.  The relevance of the Hamilton-Jacobi-Bellman equation to our control problem is explained in the following main result of this section:
\begin{thm}
Assume Hypotheses \ref{hp:grad-psi} and \ref{hp:psi} and suppose that $\lambda >0$. Then the following holds:
\begin{enumerate}
\item For all a.c.s. we have $\bJ(x,\gamma)\geq v(x)$.
\item The equality holds if and only if the following feedback law is verified by $\gamma$ and $\tX^x$:
\begin{align}\label{eq:Gamma1}
\tilde{\gamma}(t) =\Gamma(\tilde{\bx}(t), \nabla v(\tilde{\bx}(t))(I-B)Pg), \qquad \P-a.s. for a.e. t\geq 0.
\end{align}
Finally, there
exists at least an a.c.s. $\bU$ verifying \eqref{eq:Gamma1}. 
In such a system, the closed loop equation admits a solution
\begin{equation}
    \begin{cases}
        \dd \tilde{\bx}(t)= B \tilde{\bx}(t) \dd t+ (I-B)Pg\tilde{\bx}(t))\dd t + \\
        \qquad \quad (I-B)Pg\left(r(\tilde{\bx}(t),\Gamma(\tilde{\bx}(t),[\nabla v(\tX(t)) (I-B)^{1-\theta}](I-B)^\theta Pg))\dd t + \dd \tilde{W}^\bU(t) \right), 
         \ \ t\geq 0\\
        \tilde{\bx}(0)=x \in X_\eta,
    \end{cases}
\end{equation}
   and if $\tilde{\gamma}(t)=\Gamma(\tilde{\bx}(t),[\nabla v(\tilde{\bx}(t))(I-B)^{1-\theta}](I-B)Pg )$
then the couple $(\tilde{\gamma},\tilde{\bx})$ is optimal for the control problem.
\end{enumerate}
\end{thm}

\section{Aknowledgment}
The author would like to express her gratitude to Marco Fuhrman and Gianmario 
Tessitore for stimulating discussions and useful comments. 

\bibliographystyle{model1a-num-names}

\end{document}